%% file: int_trans_3d.tex
\documentclass{amsart}
\usepackage{master}
\usepackage[margin=1.4in]{geometry}

\begin{document}
\title{Integral transforms for finite gauge theories}
\author{Jackson Van Dyke}
\date{\today}
\begin{abstract}
This paper shows that quantization of $\pi$-finite spaces, as a functor out of a
higher category of spans, is equivariant in two ways:
Symmetries of a given polarization/Lagrangian always induce
coherent symmetries of the quantization. 
On the other hand, symmetries of the entire phase space a priori only induce projective
symmetries, with an invertible once-categorified theory, the anomaly theory, encoding the
projectivity.

We give projective symmetries of three-dimensional finite gauge theories a
concrete description via a twice-categorified analogue of Blattner-Kostant-Sternberg
kernels and the associated integral transforms, such as the Fourier transform. 
This establishes an analogy between certain instances of the $\pi$-finite quantization
procedure considered herein and the geometric quantization of a symplectic vector space. 
\end{abstract}
\maketitle
\tableofcontents

\section{Introduction}


Geometric quantization relates (semi)classical states to quantum ones.
Namely, the geometric quantization of a polarized symplectic vector space $V$ consists of some sort of
``linearization of half of the variables''.
For instance, given a polarization $V \simeq T^* \ell$, one model for the Hilbert space
of states is $L^2\left(\ell\right)$.
This space of states has the curious property that, even though it is most
naturally thought of as associated to the symplectic vector space $V$, it is only
\emph{projectively} equivariant for the natural symplectic group of symmetries.

The traditional way of identifying this projectivity is by constructing a natural integral
transform $T_g$ associated to each symplectic group element $g$, using what are called
Blattner-Kostant-Sternberg (BKS) kernels \cite{LV:Weil,Pol:abelian,BW:quant}.
E.g. when $g$ is the symplectic matrix sending the position coordinates in $V$ to the
momentum coordinates, the associated integral transform is the \emph{Fourier transform}.
The integral transforms don't quite compose coherently:
\begin{equation*}
c\left(g,h\right) = T_{gh}^{-1} \comp g T_h \comp T_g
\end{equation*}
for some scalar $c\left(g, h\right) \in \U\left(1\right)$. 
This defines a $2$-cocycle, and the extension classified by 
$\left[c\right] \in H^2\left(B\Sp\left(T^* \ell\right)\right)$ is essentially the metaplectic
group.

In summary: Geometric quantization is only \emph{projectively} a functor out of the
appropriate category of phase spaces. 
On the other hand, it is manifestly (linearly) functorial out of \emph{polarized} phase
spaces. 

We study an analogue of this projective (resp. linear) equivariance with respect to automorphisms
of the full phase space (resp. of the Lagrangian) in the setting of quantization of
$\pi$-finite groupoids in dimension three
\cite{K:finite,Q:finite,T:finite,F:higher,FHLT,FMT}.
Since a full treatment has not appeared in the literature, we assume a certain
\Cref{hyp:sigma} (and the cobordism hypothesis) which enables us to consider the 
\emph{quantization} of a $\pi$-finite space, possibly decorated by a cocycle $\tau$
representing a class in $H^d\left(X\right)$, which is a $d$-dimensional TQFT 
$\sigma_{X ,\tau}^d$.
For example, when $X$ is the classifying space of a finite group we obtain Dijkgraaf-Witten
theory.

Under the analogy considered herein, 
the analogue of the symplectic vector space $V = T^* \ell$ is $A = L\dsum \pdual{L}$, with
``Lagrangian'' the finite group subgroup $L$.
The symplectic form is replaced by a quadratic\footnote{There is a discrepancy between the
two sides of the analogy: The symplectic structure is replaced by an orthogonal one. This
is a version of the well-known analogy between between the Weil representation and spinors
\cite{Del:spinors}.} form on $A$ given by evaluation.
The analogue of the geometric quantization itself is the Dijkgraaf-Witten theory for $L$,
written $\sigma_{BL}^3$.

In any dimension, several different avatars of a ``projective action'' on a TQFT
are shown to be equivalent in \cite[\S 3.3]{VD:3d}:
Projective representations of $\Bord^{BG}_d$ are equivalent to 
modules over the quiche given by $(d+1)$-dimensional pure topological gauge
theory twisted by a representative for the projectivity class.
This is the main upshot of \Cref{hyp:sigma}, as is spelled out here in
\Cref{prop:sigma_mod}.

Restricting to dimension three, it is shown in \cite[\S 5]{VD:3d} that
orthogonal automorphisms of an odd-order metric group $A$ induce canonical projective
automorphisms of the theory $\sigma_{BL}^3$ which cannot necessarily be linearized.
Equivalently, certain $3$-group extensions (defined in \cite[\S 4.5]{VD:3d}) of $\O\left(A
, \ev \right)$ act linearly on $\sigma_{BL}^3$.
In other words, \textbf{twice-categorified geometric quantization of finite abelian odd-order
metric groups is projectively functorial.}
This is analogous to the linear action of the extension $\Spin$ of $\SO$ on spinors, as is
explained in \cite[\S 4.7]{VD:3d}.

This work spells out this canonical projective action 
in terms of ``topological symmetries'' \cite{FMT}. 
We identify a semiclassical automorphism (i.e. a correspondence/span)
realizing the action of each group element. This is done using a 
canonically defined $2$-cocycle on $A$, which plays an analogous role to that played by the
BKS kernels. In \Cref{sec:orthogonal_full} we explain the connection with the
classification of bimodule categories over pointed fusion categories.
The main theorems concerning this projective action (\Cref{thm:ENO,thm:FT})
can be summarized as follows (\Cref{sec:intro_O} is dedicated to a more detailed
overview):
\begin{thm*}
The projective action of $\O\left(A = L\dsum \pdual{L}\right)$ on $\sigma_{BL}^3$ can be
described via explicit bimodule categories. Explicitly, 
any element of $\O\left(A\right)$ exchanging $L$ with $\pdual{L}$ acts 
on $\sigma_{BL}^3$ via a Fourier-type transform defined in \Cref{defn:FT}. 
\end{thm*}

In contrast to the full orthogonal group, 
we show that if we only ask for automorphisms of any fixed Lagrangian $L$ in $A$,
then this induces a fully coherent linear action on the theory $\sigma_{BL}^3$.
In \Cref{sec:base} we also extend this result to the case of possibly nonabelian $L$,
twisted by a $3$-cocycle. 

The quadratic form on $A$ defines both this $2$-cocycle on $A$ (analogue of BKS kernel, as
detailed in \Cref{sec:BKS}) and a class in $H^4\left(B^2 A , \right)$ \cite{EM:q}.
Choosing a $4$-cocycle representative of the latter, the quantization of $A$
twisted by this $4$-cocycle is the so-called Crane-Yetter theory sending the point to the
category of vector spaces graded by $A$, with braiding determined by the quadratic form. 
This famously encodes the framing anomaly of Reshetikhin-Turaev theory:
The regular bimodule over the braided category defines a relative/twisted/anomalous
$3$-dimensional TQFT \cite{RT1,RT2,Wit:RT,Wal:RT,FHLT,Ben:unit,Ben:WRT_CY}.

One can now ask if a symmetry group can be gauged. 
In this paper, and in \cite{VD:3d}, the symmetry in question is the action of the
orthogonal group $\O\left(A\right)$ on Reshetikhin-Turaev theory (defined by the MTC
$\Vect\left[A\right]$) acting via braided autoequivalences. 
As is argued in \cite{VD:3d}, 
\textbf{the anomaly theory encoding the projectivity (or 't Hooft anomaly) of this symmetry can be
obtained by coupling the Crane-Yetter theory itself to an $\O\left(A\right)$-background
field}. 

We are considering the particularly simple case where the group is polarized: $A = L\dsum
\pdual{L}$. In this case the Reshetikhin-Turaev theory in question is in fact a
Turaev-Viro theory sending the point to the fusion category $\Vect\left[L\right]$.

Said differently, the fact that $\O\left(A\right)$ acts on $A = L\dsum \pdual{L}$ can be
leveraged to describe the projectivity of its action on $\sigma_{BL}^3$. 
The same phenomenon is happening in our description of the integral transforms herein:
Even though generic orthogonal group elements don't preserve $L$ within $A$, we leverage
the fact that the group tautologically acts on $A$ itself to construct the correspondences
which quantize to give the projective action of $\O\left(A\right)$ on $\sigma_{BL}^3$.

\subsection{Relationship with other work}

The same Fourier-type transform was studied in \cite{FT:ising}, where they use it to encode
Kramers-Wannier duality, with a focus on the $2$-dimensional CFTs living on the boundary
of the $3$-dimensional theory in question. In this paper we focus primarily on the
$3$-dimensional theory as the thing being acted \emph{on}, rather than the thing acting on
its boundary theories.

One of the main mathematical roles played by TQFTs is providing a source of manifold
invariants. Just as (framed, oriented, etc.) $3$-dimensional TQFTs give rise to invariants
of closed (framed, oriented, etc.) $3$-manifolds, the data of the action of a group $G$ on
a $3$-dimensional TQFT upgrades the theory: It now provides an invariant of $3$-manifolds
equipped with a $G$-bundle (sometimes with a possibly flat connection, depending on the
nature of the action of the group). 
Viewed through this lens, the results of this paper show that there is a possible anomaly
when trying to couple a $3$-dimensional TQFT/invariant to a flat $G$-bundle when $G$ is
acting via the full orthogonal group $\O\left(L\dsum \pdual{L}\right)$. On the other hand
there is no obstruction when $G$ is just acting  via $\Aut\left(L\right)$. 

Just as the mathematical theory of geometric quantization is tightly tied to the theory of
symplectic geometry and topology, the quantization of metric groups, as studied in this
paper, is closely tied to the geometry of finite orthogonal groups.
In geometric quantization, the local model is the Hilbert space of $L^2$-functions on half
of the variables. To glue these local models coherently over some global symplectic
manifold, one has to trivialize a certain $H^2$-obstruction which can be accessed via the
universal obstruction class in the cohomology of the symplectic group. 
Similarly, in the context of quantizing metric groups, the ``local model'' is a 
$3$-dimensional TQFT $\Vect\left[L\right]$. 
If we are given some base topological space, say a manifold/parameter space, with a
principal $\O\left(L\dsum \pdual{L}\right)$-bundle, the anomaly studied in this paper
pulls back to this base to give the obstruction to gluing these local models coherently
over the base parameter space.

As is discussed in \cite{FT:ising}, this domain wall between Dijkgraaf-Witten for $L$ and
$\pdual{L}$ is a form of electromagnetic duality for finite $3$-dimensional gauge theories. 
This domain wall, along with the action of the rest of the group, is encoded by an action
of a $4$-dimensional theory (or really a ``quiche'' \cite{FMT}) given by pure topological
$4$-dimensional $G$ gauge theory. These bulk $4$-dimensional theories also sit in various
dualities. These theories are finite versions of the topological twists of SUSY Yang-Mills
theory studied in the geometric Langlands program. 
The coupling of the $3$-dimensional theory to a background $G$-bundle, or equivalently to 
the bulk $4$-dimensional $G$ gauge theory, is in turn analogous to the \emph{relative}
Langlands program, and the anomalies encountered herein are analogous to the degree $4$
obstructions appearing in the relative Langlands program \cite[\S
5.1]{BZSV}, \cite[\S B]{BDFRT}, \cite{T:chiral}.

\subsubsection{Three-functor formalisms}

The existence of the theories and morphisms between them studied in this paper follows
from the cobordism hypothesis and \Cref{hyp:sigma}. The latter assumes the existence of a
functor from a category of spans of $\pi$-finite groupoids to a prescribed symmetric
monoidal $\infty$-category $\cT$. 
We interpret this as a ``quantization'' of $\pi$-finite groupoids. The study of theories
arising in this way was introduced in \cite{K:finite} and studied further in
\cite{Q:finite,T:finite}. We will consider the fully local case introduced in
\cite{F:higher} and studied in \cite[\S 3,8]{FHLT} and \cite[\S A]{FMT}.

In a different context, three-functor formalisms are functors from a category of spans 
of objects of a fixed category 
(usually of algebra geometric spaces such as varieties or stacks)
to the category of $\infty$-categories \cite{Sch:6F}.
Usually the correspondences are allowed to be
decorated by morphisms to some coefficient category. 
In other words, \Cref{hyp:sigma} asserts the existence of a $3$-functor
formalism for $\pi$-finite groupoids satisfying certain conditions. 
In these terms, \textbf{the theorems of this paper describe various (possibly projective)
equivariance properties which automatically hold for any $3$-functor formalism.}

\subsection{Outline}

In the rest of the introduction we provide 
a more detailed summary of the main results of this paper.

\Cref{sec:top_symm} is an overview of topological symmetry, mostly following \cite{FHLT,FMT}. 
Preliminaries regarding the target category and boundary/relative theories are discussed
in \Cref{sec:relative_theories}.
In \Cref{sec:sigma} we discuss quantization of finite groupoids, and introduce
\Cref{hyp:sigma}.
In \Cref{sec:mod_str} we review the notion of a module structure on a QFT from \cite{FMT}.

\Cref{sec:semiclassical} consists of a discussion of semiclassical symmetries. 
Namely, we define linear and projective semiclassical actions in \Cref{sec:semi}, and
discuss the anomaly/obstruction to extending a semiclassical action on a $\pi$-finite
space $X$ to a semiclassical action on the decorated space $\left(X , \tau\right)$ for
some cocycle $\tau$.

In \Cref{sec:base}, we show that automorphisms of a finite group $L$ preserving a
$3$-cocycle $\tau$ act canonically on the theory $\sigma_{BL , \tau}^3$. 
In \Cref{sec:f} we define the action up to isomorphism $f$, in \Cref{sec:m} we define a system
of products $m$, and in \Cref{sec:a} we define an associator.
In \Cref{sec:base_O3} we discuss the obstruction $O_3\left(f\right)$, and in
\Cref{sec:base_O4} we discuss the obstruction $O_4\left(f , m\right)$. 
In \Cref{sec:coherent_base} we state and prove \Cref{thm:base}.
The connection with semiclassical actions is discussed in 
\Cref{sec:semi_3d}.

In \Cref{sec:orthogonal} we instead consider automorphisms of the full group $L\dsum
\pdual{L}$ which preserve the quadratic form $\ev$.
In \Cref{sec:BKS} we define an analogue of the BKS kernel.
The quantization of this, as a domain wall between TQFTs, is considered in
\Cref{sec:intertwining}.
In \Cref{sec:orthogonal_pt}, the local data defining the domain wall (a bimodule category)
is discussed. This is related to the constructions in \cite{ENO:homotopy} in
\Cref{sec:orthogonal_full}, to give a projective action of the full finite orthogonal
group in \Cref{sec:orthogonal_TFT}.
\Cref{sec:schur} is dedicated to a discussion of some technical points regarding the Schur
isomorphism, which are used in \Cref{sec:orthogonal}.

\subsection{Acknowledgements}
I warmly thank my advisor David Ben-Zvi for his constant
help and guidance, without which this project would not exist. 
I would also like to thank Dan Freed, David Jordan, 
Lukas M\"uller, German Stefanich, Rok Gregoric, and 
Will Stewart for helpful discussions. 

Part of this project was completed while the author was supported by 
the Simons Collaboration on Global Categorical Symmetries.

This research was also partially completed while visiting the Perimeter Institute for
Theoretical Physics.
Research at Perimeter is supported by the Government of Canada through the Department of
Innovation, Science, and Economic Development, and by the Province of Ontario through the
Ministry of Colleges and Universities.

\subsection{Orthogonal symmetries}
\label{sec:intro_O}

Symmetries of the phase space which do not preserve a chosen Lagrangian act on the 
associated geometric quantization in a subtle way. One way of accessing the action is via
the theory of integral transforms associated to BKS kernels. 
In \Cref{sec:orthogonal} we develop an analogue of
this theory in the following context:
The phase space is a polarized metric group $L\dsum \pdual{L}$ equipped with evaluation,
and the symmetries are the elements of $\O\left(L\dsum \pdual{L}\right)$. 

In particular, for any two Lagrangians $L$ and $K$ in a metric group $A$, we consider the
correspondence diagram:
\begin{equation*}
\begin{cd}
& A / \left(L\cap K\right) 
\ar{dr}{ / \left(L / L\cap K\right)} 
\ar{dl}{ / \left(K / L\cap K\right)} &\\
A / K && A / L
\end{cd}
\end{equation*}
We make use of this diagram when $K$ is the image of $L$ under some $g\in
\O\left(A\right)$, and $L$ and $gL$ are transverse. This then becomes the
correspondence:
\begin{equation*}
\begin{cd}
& A \ar{dl}\ar{dr} & \\
L && gL
\end{cd}
\end{equation*}
This diagram defines a $1$-morphism in the category of $\pi$-finite groupoids with 
$1$-morphisms given by correspondences of such, 
$2$-morphisms given by correspondences of correspondences, and so on. 
In \Cref{sec:sigma} we consider the particular flavor of the correspondence category where
the groupoids can be decorated by $\units{\bk}$-cocycles (where $\bk$ is a fixed
algebraically closed field of characteristic zero). 
Cohomologically distinguishable $2$-cocycles on $A$
define different $1$-morphisms in this category. 

\Cref{sec:sigma} discusses a very general construction of TQFTs from $\pi$-finite
groupoids depending on \Cref{hyp:sigma}, which assumes that there is an appropriate
``quantization'' functor. This functor sends $\pi$-finite groupoids to objects of the
chosen target category, and correspondences to morphisms.
For target category $\Fus$, the Morita $3$-category of $\bk$-linear fusion categories,
this functor sends the groupoid $BL$ to the fusion category $\Vect\left[L\right]$, and
correspondences get sent to bimodule categories. 
Precomposing this functor with the mapping space functor $\Map\left(- , BL\right)$ from
$\Bord_3$ to the correspondence category, we obtain the TQFT $\sigma_{BL}^3$, untwisted
Dijkgraaf-Witten theory for the group $L$.
Correspondingly, decorated correspondences define domain walls between the associated
TQFTs.

When $A\simeq L\dsum \pdual{L}$ there is a canonical class in 
$H^2\left(A ,\units{\bk}\right)$. Write $\kappa_{\ev}$ for 
any $2$-cocycle representative of this class. The $1$-morphism in the correspondence
category 
\begin{equation*}
\begin{cd}
& \left(A , \kappa_{\ev}\right) \ar{dl}\ar{dr} & \\
L && gL
\end{cd}
\end{equation*}
induces the integral transform from the $3$-dimensional TQFT associated to $L$,
$\sigma_L^3$, and the one associated to $gL$, $\sigma_{gL}^3$. I.e. $\kappa_{\ev}$ is the
higher kernel of the integral transform defining the domain wall between these theories.

If $g$ is an orthogonal homomorphism which sends $L$ to $\pdual{L}$, then the associated
domain wall can be thought of as a higher Fourier transform, as defined in \Cref{defn:FT}:
\begin{equation*}
\FT \colon \sigma_{BL}^3 \to \sigma_{B\pdual{L}}^3
\end{equation*}
The assignment to the point is studied in \Cref{sec:orthogonal_pt}.
As a $\Vect\left[L\right]$-$\Rep\left(L\right)$-bimodule category this 
is a Morita equivalence, as studied in \cite{EGNO:TC,FT:ising}, 
which only has a single simple object, i.e. is $\Vect$.
Note that this exchanges convolution on $\Vect\left[L\right]$ with the symmetric-monoidal
structure on $\Rep\left(L\right)$, as the usual Fourier transform does.

After making a noncanonical choice, this is written as a
$\Vect\left[L\right]$-$\Vect\left[L\right]$-bimodule in \Cref{prop:FT_VectVect}, which
again only has a single simple object and is of order $2$.

In \Cref{sec:orthogonal_full} we establish a projective action of the full orthogonal
group $\O\left(L\dsum \pdual{L}\right)$ on $\sigma_{BL}^3$ by explicitly defining a 
$1$-morphism in the correspondence category for every element $g\in \O\left(L\dsum
\pdual{L}\right)$. 
This uses the classification of bimodule categories over pointed fusion categories in
\cite{ENO:homotopy}. 
These $1$-morphisms induce domain walls between the
different models for the same $3$-dimensional TQFT.

There is a certain element of $\O\left(L\dsum \pdual{L}\right)$ of order $2$, which sends
$L$ to $\pdual{L}$. The construction of a bimodule category from \cite{ENO:homotopy} from
this element is shown to agree with the assignment to the point of the twice-categorified
Fourier transform (\Cref{defn:FT})
in \Cref{thm:ENO}.
This is then cast in field-theoretic terms in 
\Cref{thm:FT}:

\begin{thm*}
Assume \Cref{hyp:sigma} and that $L$ has odd order.
The canonical projective action of $T_\sigma\in \O\left(L\dsum \pdual{L}\right)$ 
on $\sigma_{BL}^3$ is via the domain wall $\FT$ (\Cref{defn:FT}).
\end{thm*}

\subsection{Automorphisms of the base}

Now we will consider an action of a finite group $G$ on $L$ preserving the class of
$\tau$, which will turn out to define a canonical action of $G$ on $\sigma_{BL,\tau}^d$.
But first we will make some general considerations regarding actions of groups on 
semiclassical data. 

\subsubsection{Semiclassical action}
\label{sec:intro_semi}

Following ideas in \cite{FMT}, we define a \emph{semiclassical action} 
(\Cref{defn:semi}) of a finite group
$G$ on semiclassical data $\left(X ,\tau\right)$ (where $\tau$ represents a class in
$H^d\left(X , \units{\bk}\right)$)
to be a correspondence (i.e. $1$-morphism)
between the trivial $\pi$-finite space to $BG$ such that when we pull back along $\pt \to
BG$ we obtain $\left(X , \tau\right)$:
\begin{equation*}
\begin{cd}
&&\left(X , \tau\right) \ar{dl}\ar{dr}&&\\
& * \ar{dr}\ar{dl} && 
\left(C_G , \tau_{G} \right) \ar{dr}\ar{dl}& \\
* && 
BG
&& *
\end{cd}
\end{equation*}
I.e. this consists of an $X$-bundle over $BG$ and an extension of $\tau$ to the total
space $C_G$. 

The semiclassical action is \emph{projective} if $BG$ is replaced by $\left(BG , c\right)$ for some
$\left(d+1\right)$-cocycle on $BG$, and a trivialization defines a correspondence
($1$-morphism) from $\left(BG , c\right)$ to $BG$ (\Cref{defn:proj_semi}).
Composing the projective action with this correspondence ``linearizes'' the projective
action to an ordinary one.

Assuming we have a semiclassical action of $G$ on $X$, 
we would like to identify the obstruction to extending this to a semiclassical action
of $G$ on $\left(X , \tau\right)$. 
I.e. we have an $X$-bundle over $BG$ with total space $C_G$, and we seek
the obstruction to extending $\tau$ from $X$ to $C_G$.

The Serre spectral sequence provides a general method for answering this question.
In \Cref{thm:ss}, under some assumptions, we show that the spectral sequence
obstruction encodes an anomaly/anomalous symmetry of the theory $\sigma_{X , \tau}^d$. 
\Cref{thm:ss} generalizes the 't Hooft anomalies studied in
\cite{Mul:thesis,MS:anomalies,KT:discrete_anomalies} to the fully-extended setting.

In \Cref{cor:ss}, we note that the anomaly/projectivity agrees with (a restriction of) the
top $k$-invariant of the higher automorphism groupoid. 
\Cref{cor:ss} generalizes part of the arguments made in \cite[\S A]{ENO:homotopy} to
arbitrary dimensions.

\subsubsection{Three dimensions}

Now we specialize to three dimensions and consider the particular space $X = BL$ for $L$
a finite group with $\tau$ a group cocycle representing a class in
$H^3\left(BL , \units{\bk}\right)\simeq H^3_{\text{Grp}}\left(L , \units{\bk}\right)$.
Let $G$ be a finite group, and consider a group homomorphism 
\begin{equation*}
\phi \colon G \to \Aut\left(L , \tau\right)
\end{equation*}
where $\Aut\left(L , \tau\right)$ denotes the group of group automorphisms of $L$ which
preserve the cohomology class of $\tau$.

The fully-extended TQFT
\begin{equation*}
\sigma_{L , \tau} \colon \Bord^\fra_3 \to \Fus 
\end{equation*}
(see \Cref{sec:top_symm}) 
sends the point to 
\begin{equation*}
\cC = \Vect\left[L\right]^\tau \ ,
\end{equation*}
the fusion category of vector spaces graded by $L$, with convolution product twisted by
$\tau$. 

Therefore, one avatar of a group acting on a category might be a map:
\begin{equation*}
\mathbf{f} \colon G \to \Aut_{\Fus}\left(\cC\right) \ .
\end{equation*}
In \Cref{sec:base}, we define such a map by explicitly defining the map on the underlying
ordinary groups, and trivializing the obstructions in \cite{ENO:homotopy}.
Any functor 
\begin{equation*}
BG \to B\Aut_\Fus\left( \sigma_{BL,\tau}^3\left(\pt\right) \right)
\end{equation*}
is equivalent to a $\left(\sigma_{BG}^{4} , \rho\right)$-module structure on  
$\sigma_{BL , \tau}^3$ by \Cref{hyp:sigma} and the cobordism hypothesis. 
(See \Cref{prop:sigma_mod}.)
This shows the following, which 
is stated as \Cref{thm:base} in the body of this paper.

\begin{thm*}
Assume \Cref{hyp:sigma}.
Let $L$ and $G$ be finite groups.
Let $\tau \in Z^3_{\text{Grp}}\left(L , \units{\bk}\right)$ 
be a $3$-cocycle in the group cohomology of $L$.
Then any group homomorphism 
\begin{equation*}
\phi \colon G\to \Aut\left(L , \tau\right)
\end{equation*}
canonically defines a
$\left(\sigma_{BG}^{4} , \rho\right)$-module structure on  $\sigma_{BL , \tau}^3$. 

This canonical module structure can be twisted by a triple
$\left(\gamma , \mu , \alpha\right)$ where 
\begin{align*}
\gamma \in C^1_{\text{Grp}} \left(G ,C^2_{\text{Grp}} \left(L, \units{\bk}\right) \right)
&&
\mu \in Z^2_{\text{Grp}}\left(G , C^1_{\text{Grp}} \left(L , \units{\bk} \right)\right)
&&
\alpha \in Z^3\left(G , \units{\bk}\right)
\ .
\end{align*}
satisfy:
\begin{align*}
d_L\left(\gamma\left(g\right)\right) = \phi\left(g\right)\left(\tau\right)  \tau^{-1}  &&
d_L \mu = d_G \gamma^{-1} 
\end{align*}
\end{thm*}

As is shown in \cite[\S 11]{ENO:homotopy}, the existence of such a functor 
\begin{equation*}
\left(f,m,a\right) \colon BG \to B\Aut_\Fus\left(\sigma_{BL , \tau}^3\left(\pt\right)\right)
\end{equation*}
implies the existence of an extension of $\tau$ to a group $3$-cocycle $\tau_G$ 
on $L\rtimes_\phi G$. 
The quantization of the correspondence
\begin{equation*}
\begin{cd}
& \left(C_G , \tau_G\right)\ar{dl}\ar{dr} & \\
BG && \pt
\end{cd}
\end{equation*}
defines a $\left(\sigma_{BG}^4 , \rho\right)$-module structure on $\sigma_{BL , \tau}^3$.
This module structure equivalently defines a functor
\begin{equation*}
BG \to B\Aut_{\Fus}\left(\Vect\left[L\right]^\tau\right) \ ,
\end{equation*}
which provides an alternative description of the map $\left(f,m,a\right)$.
This is discussed in more detail in \Cref{sec:semi_3d}.

\input{topological_symmetry}

\input{semiclassical}
\input{base}

\input{orthogonal}


\appendix
\input{schur}


\bibliographystyle{amsalpha}
\bibliography{references}

\end{document}

%% file: topological_symmetry.tex
\section{Topological field theory}
\label{sec:top_symm}

Let $d \in \ZZ^{\geq 0}$.
Given a symmetric-monoidal $\left(\infty , d+1\right)$-category $\cT$, which will be the
target for our theories, write $\cT^\sim$ for the groupoid obtained by discarding all non-invertible
morphisms at all levels.
We will assume $\cT$ has duals.\footnote{Otherwise replace $\cT$ with the subcategory consisting of
the fully-dualizable objects of $\cT$.}
We will also assume that $\Om^{d+1} \cT = \bk$ for $\bk$ an algebraically closed field of
characteristic zero.
Throughout the paper, we will write $\Vect$ for the $\bk$-linear category of
finite-dimensional vector spaces over $\bk$.

\begin{rmk}
Besides having duals and satisfying $\Om^{d+1} \cT = \bk$, we will occasionally assume
that $\cT$ satisfies an additional hypothesis.
In particular, when we discuss module structures in the sense of \cite{FMT}, we will assume
that the target $\cT$ is sufficiently additive to support quantization of $\pi$-finite
spaces in the sense of \cite[\S 8]{FHLT} and \cite[\S A]{FMT} (this is \Cref{hyp:sigma}).
\end{rmk}

\subsection{Relative/boundary theories}
\label{sec:relative_theories}

A symmetric monoidal functor
\begin{equation*}
\al \colon \Bord_{d}^X \to \cT
\end{equation*}
is a \emph{once-categorified $d$-dimensional $X$-TQFT}, 
where $\cT$ is the fixed target from the beginning of \Cref{sec:top_symm}. Let 
\begin{equation*}
1_X \colon \Bord_{d}^X \to \cT
\end{equation*}
denote the trivial once-categorified $d$-dimensional $X$-TQFT.
Recall the notion of a relative theory \cite{FT:relative}.
These are also called \emph{twisted theories} \cite{ST:twisted,JFS}.

\begin{defn*}
A theory defined right-relative to $\al$ is a lax natural transformation (in the sense
of \cite{JFS})
\begin{equation*}
\al \to 1
\ .
\end{equation*}
A theory defined left-relative to $\al$ is a lax natural transformation
\begin{equation*}
1 \to \al \ .
\end{equation*}
\end{defn*}

\begin{rmk}
Recall the definition of lax (resp. oplax) natural transformations from \cite{JFS}. 
Consider the arrow categories $\cT^\down$ and $\cT^\to$, and the
source and target functors $s,t \colon \cT^* \to \cT$
for $* = \down , \to$.
Following \cite{JFS}, a lax (resp. oplax) natural transformation $\al \to 1$ is a functor 
\begin{align}
F_\al \colon \Bord^X_d \to \cT^{\down}
&&
\left(
\text{resp. }
F_\al \colon \Bord^X_d \to \cT^{\to}
\right)
\end{align}
such that $s\comp F_\al = \al$, and 
$t\comp F_\al = 1$.

Throughout, we will use the lax version, as written in the above definition of relative
theories. 
The reason we use the lax version, as noted in \cite[Example 7.3]{JFS}, that the
\emph{lax} natural transformations
from the trivial theory to itself consist of theories of dimension lower (\cite[Theorem
7.4]{JFS}) whereas the same is not true when lax is replaced with oplax.
We need the analogous result for $X$-theories in order to
establish a trivialized anomalous theory $1\lto{\sim}\al\to 1$ as an honest theory of
one dimension lower.

Also noted in \cite[Example 7.3]{JFS}, is the fact that oplax natural transformations are
``elements'' $F_{\sigma}\left(M\right) \colon 1 \to \al\left(M\right)$ for $M$ a
\emph{closed} bordism of any codimension, which is for example the point of view taken in
\cite{FT:relative}.
Oplax natural transformations are also used in \cite{FT:gapped}.
The difference does not appear in the present work: All results still hold upon replacing
all lax natural transformations with oplax ones.
\label{rmk:lax}
\end{rmk}

\begin{rmk}
Often a once-categorified $d$-dimensional theory $\al$
extends to a $\left(d+1\right)$-dimensional theory:
\begin{equation*}
\begin{tikzcd}
& \cT \\
\Bord_d^X \ar[hook]{r}\ar{ur}{\al}&
\Bord_{d+1}^X \ar[dashed]{u}
\end{tikzcd}
\end{equation*}
In this case, a relative theory $\al \to 1$ can be upgraded to what is called a
\emph{boundary theory}.
Boundary theories are defined as functors out of the extended bordism category
$\Bord^{X , \p}_{d+1}$, described in \cite[\S 4.3]{L:CH}.
The connection with the notion defined here is made in \cite[Theorem 7.15]{JFS}.
\label{rmk:boundary}
\end{rmk}

\subsection{Projective/anomalous theories}
\label{sec:proj}

Recall the projectivization $\PP \Om \cT$ of the target category as defined in
\cite{VD:3d}.

\begin{defn}
A $d$-dimensional \emph{projective $\left(X ,
\z\right)$-TQFT} is a non-zero symmetric-monoidal functor:
\begin{equation*}
\overline{F} \colon \Bord^{\left(X , \z\right)}_{d} \to \PP\Om \cT
\ .
\end{equation*}
Recall by construction of $\PP \Om \cT$ there is a canonical `target functor'
(forgetting the module) $t \colon \PP \Om \cT \to \units{\cT}$.
Given such a theory $\overline{F}$, composing with the functor $t$
results in an invertible theory:
\begin{equation*}
\al \colon \Bord^{\left(X , \z\right)}_{d} \lto{\overline{F}}
\PP\Om \cT \lto{s}
\units{\cT}
\ .
\end{equation*}
This theory is the \emph{projectivity} of the projective theory $\overline{F}$.
Sometimes to emphasize that a theory is \emph{not} projective, we will call it
\emph{linear}.
\end{defn}

Passing to the assignment to the point, we obtain the space of projective families of
objects of $\Om \cT$ over $X$, $\Hom\left(X , \PP\Om \cT\right)$. The target functor 
$t \colon \PP \Om \cT$ induces a map to the space of families of invertible objects of
$\cT$ over $X$:
\begin{equation*}
\Hom\left(X , \PP \Om \cT\right) \to \Hom\left(X , \units{\cT}\right)
\end{equation*}
The fiber over any particular projectivity $\tau \colon X\to \units{\cT}$, which we will
write as $\Hom\left(X , \PP\Om \cT\right)_{\tau}$, consists of projective families of
objects of $\Om \cT$ with projectivity $\tau$.

\subsection{TQFTs associated to \texorpdfstring{$\pi$}{pi}-finite spaces}
\label{sec:sigma}

Let $X$ be a space (i.e. higher groupoid) which is (connected, pointed, and)
$\pi$-finite\footnote{This means $X$ has finitely many homotopy groups, each of which is
finite.}.
There is a recipe for constructing a TQFT using $X$, which was
introduced in \cite{K:finite} and studied further in \cite{Q:finite,T:finite}.
We will consider the fully local case introduced in \cite{F:higher} and studied in
\cite[\S 3,8]{FHLT} and \cite[\S A]{FMT}.
A formal treatment of these theories is the subject of an 
upcoming work of Claudia Scheimbauer and Tashi Walde.

\subsubsection{The quantization map}

We will proceed heuristically, following \cite[\S 3]{FHLT} and \cite[\S A.2]{FMT}, to fix
notation and describe expectations which will eventually be stated and assumed in
\Cref{hyp:sigma}.

Let $\Fam_{d+1}$ denote the category with objects finite $\left(d+1\right)$-groupoids, 
$1$-morphisms given by correspondences of $\pi$-finite spaces, 
$2$-morphisms given by correspondences of correspondences, and so on until level
$\left(d+1\right)$. 
(Two $\left(d+1\right)$-morphisms are regarded as identical if they are equivalent.)

Let $\cT$ be the arbitrary symmetric monoidal target with duals, fixed at the beginning of
\Cref{sec:top_symm}.
Let $Y$ be an object of $\Fam_{d+1}$.
A local system on $Y$ valued in $\cT$ is a functor $Y \to \cT$. 
Write $\Fam_{d+1}\left(\cT\right)$ for the category of $\pi$-finite spaces equipped with a
local system valued in $\cT$.

For example, $\Fam_{d+1}\left(B^{d+1} \units{\bk}\right)$ has objects given by pairs
$\left(Y , \tau\right)$, where $\tau$ is a cocycle 
\begin{equation*}
\tau \colon Y \to B^{d+1} \units{\bk}
\end{equation*}
representing a class in $H^{d+1}\left(Y , \units{\bk}\right)$.
A morphism is a correspondence:
\begin{equation}
\begin{cd}
& \left(E , \mu\right) \ar{dr}{p_2}\ar{dl}{p_1} & \\
\left(Y_1 , \tau_1\right)  && \left(Y_2 , \tau_2\right)
\end{cd}
\label{eqn:corr_mor}
\end{equation}
where $\mu \colon E \to B^d \units{\bk}$ satisfies
\begin{equation}
d \mu = \left( p_1^* \tau_1 \right)^{-1}  \cdot 
\left( p_2^* \tau_2 \right)
\ .
\label{eqn:triv}
\end{equation}
$2$-morphisms are correspondences between correspondences with a similar condition on the
cocycles, and so on to define morphisms up to level $\left(d+1\right)$.

Recall we assumed $\Om^{d+1}\cT = \bk$. Therefore there is a natural functor:
$B^{d+1} \units{\bk} \to \cT$ inducing a functor:
\begin{equation}
\Fam_{d+1}\left(B^{d+1} \units{\bk}\right) \to \Fam_{d+1}\left(\cT\right) \ .
\label{eqn:twist_loc_sys}
\end{equation}

We will assume that the following holds.

\begin{custom}{Hypothesis Q}
For $\cT$ appropriately additive\footnote{As is explained in \cite[\S 3]{FHLT}, for $X$ an
ordinary groupoid, then $\cT$ must be additive in the sense that the colimit $\dirlim_{x\in
X}\tau\left(x\right)$ in $\cT$ exists, and agrees with the limit $\invlim_{x\in X}
\tau\left(x\right)$, where we are regarding $\tau$ as defining a $\cT$-valued local system
on $X$ by \eqref{eqn:twist_loc_sys}.},
there is a quantization\footnote{See \cite[Remark A.7.1]{FMT} where it is explained how
$\Sum_{d+1}\left(X\right)$ (and therefore $\sigma_X^{d+1}$) can be obtained by 
integrating over fluctuating fields.}
functor
\begin{equation}
\Sum_{d+1} \colon \Fam_{d+1}\left(B^{d+1} \units{\bk}\right) \to \cT
\label{eqn:sum}
\end{equation}
such that there is an invertible natural transformation between 
\begin{equation*}
\left(X , \tau\right) \mapsto \Hom_\cT \left(1 , \Sum_{d+1}\left(X , \tau\right)\right)
\end{equation*}
and 
\begin{equation*}
\left(X , \tau\right) \mapsto 
\Hom\left(X , \PP\Om \cT\right)_{\tau}
\end{equation*}
viewed as functors out of $\Fam_{d+1}\left(B^{d+1}\units{\bk}\right)$, where 
$\Hom\left(X , \PP\Om \cT\right)_{\tau}$ is defined in \Cref{sec:proj}.
\label{hyp:sigma}
\end{custom}

\begin{rmk}
Unpacking the existence of the natural transformation in \Cref{hyp:sigma}, we see that it
ensures that we have equivalences 
\begin{equation}
\Hom_{\cT}\left(1 , \Sum_{d+1}\left(X\right)\right) \simeq
\Hom\left(\hofib\left(\tau\right) , \Om \cT\right)
\label{eqn:hyp:sigma}
\end{equation}
for all objects $\left(X , \tau\right)$ of $\Fam_{d+1}\left(B^{d+1}\units{\bk}\right)$,
and given a morphism from $\left(X_1 , \tau_1\right)$ to $\left(X_2 , \tau_2\right)$ in
$\Fam_{d+1}\left(B^{d+1} \units{\bk}\right)$, we have a commuting diagram:
\begin{equation*}
\begin{tikzcd}
\Hom_\cT\left(1 , \Sum_{d+1}\left(X_1\right)\right)
\ar[phantom]{r}{\simeq}\ar{d}&
\Hom\left(\hofib\left(\tau_1\right) , \Om \cT\right) \ar{d} \\
\Hom_\cT\left(1 , \Sum_{d+1}\left(X_2\right)\right)
\ar[phantom]{r}{\simeq} & 
\Hom\left(\hofib\left(\tau_2\right) , \Om \cT\right)
\end{tikzcd}
\end{equation*}
I.e. higher representations of $\Om X$ are equivalent to morphisms from $1$ to
$\Sum_{d+1}\left(X\right)$ in $\cT$.
This generalizes the classical fact about modules over the group algebra and
$G$-representation, as in \Cref{exm:sigma_BG} \ref{sigma2_BG}.
\label{rmk:hyp:sigma}
\end{rmk}

\begin{rmk}
In \cite[\S 8.2]{FHLT}, the map $\Sum_n$ is constructed at the level of
objects and shown to be a functor up to $2$-morphisms.
In particular, the cases in \Cref{exm:sigma_BG}
are worked out in \cite[\S 8.1,8.3]{FHLT}.
\end{rmk}

\begin{exm}
\begin{enumerate}[label = (\roman*)]
\item Let $\cT = \Vect$. 
Then $\Sum_1\left(X\right) = \bk\left(\pi_0 X\right)$ is the vector space of $\bk$-valued
functions on $\pi_0 X$. \Cref{hyp:sigma} is satisfied, since 
the natural transformation \eqref{eqn:hyp:sigma} required in \Cref{hyp:sigma} is:
\begin{equation*}
\Hom_{\Vect}\left(1 , \Sum_1\left(X\right)\right) \simeq \Hom\left(X , \bk\right) \simeq 
\Hom\left(\pi_0 X , \bk\right) \ .
\end{equation*}

\item\label{sigma2_BG} Set $X = BG$ for a finite group $G$, and let $\cT = \Alg$ be the Morita $2$-category
of algebras.
Let $\Sum_2\left(BG\right) = \bk\left[G\right]$ be the group algebra.
\Cref{hyp:sigma} is satisfied, since 
the natural transformation \eqref{eqn:hyp:sigma} required in \Cref{hyp:sigma} is:
\begin{equation*}
\Hom_\Alg\left(1 , \Sum_2\left(BG\right)\right) = \lMod{\bk\left[G\right]} \simeq
\Rep\left(G\right) \simeq \Hom\left(BG , \Vect\right) \ .
\end{equation*}

We can equip $BG$ with a $2$-cocycle $\mu$ (i.e. a group $2$-cocycle)
which twists the convolution structure on the group algebra, resulting in 
$\Sum_2\left(BG , \tau\right)$.

\item \label{fus} Set $X = BG$.
Let $d = 2$, and take $\cT$ to be the Morita $3$-category of monoidal categories.
Then $\Sum_3\left(BG\right)$ is $\Vect\left[G\right]$, the fusion category of
vector bundles on $G$ with convolution. 
We can equip $BG$ with a $3$-cocycle $\al$ which twists the fusion structure on
$\Vect\left[G\right]$, yielding $\sigma_{BG,\al}^3\left(\pt\right) =
\Vect\left[G\right]^\al$.
This is \cite[Example A.65]{FMT}.

\item \label{bimodule_cats} Continuing the previous example, $\Sum_3$ can be extended to certain $1$-morphisms
in $\Fam_3\left(B^3 \units{\bk}\right)$ as follows.
Consider a $1$-morphism in $\Sum_3$ of the form
\begin{equation*}
\begin{tikzcd}[sep=tiny]
& \left(BC , \kappa\right) \ar{dl}\ar{dr}& \\
BG && BH
\end{tikzcd}
\end{equation*}
where $C\subset G\dsum H$ is a subgroup, and 
$\kappa$ is a $2$-cocycle on $C$. 
The alternating bicharacter $\alt \kappa$ on $C$ 
(see \Cref{sec:schur}) classifies a $\Vect\left[G\right]$-$\Vect\left[H\right]$-bimodule
category, as explained in \cite[\S 2.7]{ENO:homotopy}, and $\Sum_3$ sends this
$1$-morphism $\left(BC , \kappa\right)$ to this bimodule category. 

\begin{itemize}
\item E.g. if $C = G\times H$ is the whole group then, for any bicharacter, the
corresponding bimodule has a single simple object.
\item E.g. if $G = H$ and $C$ is the graph of a group automorphism of $G$, the bimodule
category is the identity bimodule $\Vect\left[G\right]$ with action twisted by the group
automorphism. 
\end{itemize}

\item\label{sigma_B2G} Set $X = B^2 G$ (where now $G$ is necessarily abelian).
Let $d = 3$ and take $\cT$ to be the Morita $4$-category of braided monoidal
categories, $\BrFus$ (see \cite[\S 3]{BJS:dualizable}).
Consider a cocycle $\tau \colon B^2 G \to B^4\units{\bk}$.
It is a classical theorem of Eilenberg-Maclane \cite{EM:q}
that cohomology classes in $H^4\left(B^2 G , \units{\bk}\right)$ correspond to quadratic
forms $G \to  \units{\bk}$.
Write $q_\tau$ for the form corresponding to $\left[\tau\right]$.
This defines a symmetric bicharacter on $G$:
\begin{equation*}
\lr{g , h}_\tau \ceqq \frac{q_\tau\left(g+h\right)}{q_\tau\left(g\right) q_\tau\left(h\right)}
\ .
\end{equation*}
Then $\Sum_4\left(B^2 G\right)$ is $\Vect\left[G\right]$ with convolution, and with
braiding specified on simples by:
\begin{equation*}
\beta_\tau \colon\bk_g * \bk_h = \bk_{gh} \lto{\lr{g,h} \id_{\bk_{gh}}} \bk_{gh} =
\bk_{hg} = \bk_h * \bk_g \ .
\end{equation*}

\item Let $X = BG$ and $d = 3$.
If $G$ is abelian, and we take $\cT = \BrFus$, then $\Sum_4\left(BG\right)$ is the
symmetric (so in particular braided) fusion category $\Vect\left[G\right]$. 

\item Let $X = BG$ and $d = 3$, only now
we would like to use the category of monoidal $2$-categories. Then 
$\Sum_4\left(BG\right)$ is the fusion $2$-category (in the sense of \cite{DR:2Fus}) of 
$G$-graded $2$-vector spaces \cite[Construction 2.1.13]{DR:2Fus}.
The dualizability of fusion $2$-categories is the subject of 
\cite[Theorem C]{Dec:dualizable_2Fus}.
\label{sigma4_BG}
\end{enumerate}
\label{exm:sigma_BG}
\end{exm}

\subsubsection{TQFTs from the quantization map}
\label{subsec:sigma}

The upshot of assuming \Cref{hyp:sigma} is that,
for a fixed object $\left(X , \tau\right)$ of $\Fam_{d+1}\left(B^{d+1}\units{\bk}\right)$, 
\eqref{eqn:sum} can be composed with the mapping space functor\footnote{Note that
\eqref{eqn:twist_loc_sys} allows us to construct a $\cT$-valued local system from the
cocycle $\tau$.} to obtain the theory 
\begin{equation*}
\sigma_{X , \tau}^{d+1} \colon 
\Bord^{\fra}_{d+1} \lto{\pi_{\leq d+1} \Map\left(- , X\right)}
\Fam_{d+1}\left(\cT\right)\lto{\Sum_{d+1}} \cT \ .
\end{equation*}

\begin{rmk}
As is remarked in \cite[\S A]{FMT} and \cite[\S 3,8]{FHLT},
$\sigma_{X , \tau}^{d+1}$ can be upgraded to an oriented theory, and if $\tau$ is trivial
then it can even be upgraded to an unoriented theory. 
We will work with the underlying framed theories in this paper.
\end{rmk}

The following results follow directly from \Cref{hyp:sigma} and the cobordism hypothesis.

\begin{prop}
Assume \Cref{hyp:sigma}. A  morphism from $\left(X_1 , \tau_1\right)$ to $\left(X_2 ,
\tau_2\right)$ in 
\begin{equation*}
\Fam_{d+1}\left(B^{d+1}\units{\bk}\right)
\end{equation*}
(i.e. correspondence as in
\eqref{eqn:corr_mor}) induces a bimodule (i.e. domain wall)
\begin{equation*}
\sigma^{d+1}_{X_1 , \tau_1} \to \sigma^{d+1}_{X_2 , \tau_2} \ .
\end{equation*}
\label{prop:corr_mor}
\end{prop}

\begin{prop}
Assume \Cref{hyp:sigma}. Every boundary theory for $\sigma_{X , \tau}^{d+1}$ is
classified by a symmetric-monoidal functor:
\begin{equation*}
\Bord_{d}^{X} \to \PP\Om \cT
\ .
\end{equation*}
If $\tau$ is trivial, then the boundary theories are classified by symmetric-monoidal
functors:
\begin{equation*}
\Bord_d^{X} \to \Om \cT \ .
\end{equation*}
\label{prop:sigma_mod}
\end{prop}

\subsection{Module structures}
\label{sec:mod_str}

We summarize the material used in our construction from \cite{FMT}.
See \cite[\S 3]{FMT} for a more detailed discussion of these definitions.

Let
\begin{equation*}
\sigma \colon \Bord_{d+1}^\fra \to \cT
\end{equation*}
be a $\left(d+1\right)$-dimensional TQFT valued in the fixed target $\cT$ from the
beginning of \Cref{sec:top_symm}.
Recall the following definition in \cite{FMT}.

\begin{defn*}
A \emph{$d$-dimensional quiche} is a pair $\left(\sigma , \rho\right)$ in which 
$\rho$ is a right topological boundary theory (or \emph{right $\sigma$-module}), which we will write as 
$\rho\colon \sigma \to 1$.
\end{defn*}

\begin{rmk}
All of the quiches considered in this paper will be of the form
$\left(\sigma_{X ,\tau} , \rho_{X , \tau}\right)$, where $\sigma_{X,\tau}$ is 
the theory associated to $\left(X , \tau\right)$ as in \Cref{sec:sigma}, and 
the (right) boundary theory $\rho_{X,\tau}$ is the natural transformation induced by the
correspondence diagram:
\begin{equation*}
\begin{cd}
& \pt \ar{dr}\ar{dl} & \\
\left(X , \tau\right) && \pt 
\end{cd}
\end{equation*}
as in \Cref{prop:corr_mor}.
Given a pointed space $X$, we will always write $\rho_{X , \tau}$ for this boundary theory. 
\label{rmk:rho_X}
\end{rmk}

A quiche is an abstract symmetry datum, in the same sense as an algebra.
The following definition is the analogue of a module, i.e. a realization of the 
quiche as symmetries of a given theory. 

\begin{defn*}
Let $\left(\sigma , \rho\right)$ be a $d$-dimensional quiche, and let $F$ be a
$d$-dimensional TQFT.
A $\left(\sigma , \rho\right)$-module structure on $F$ is the pair $\left(F_\sigma ,
\theta\right)$ where $F_{\sigma}$ is a (left) boundary theory
$F_\sigma \colon 1 \to \sigma$ which is equipped with an isomorphism of
$d$-dimensional theories:
\begin{equation*}
\theta \colon \rho\tp_\sigma F_\sigma \lto{\sim} F \ .
\end{equation*}
\end{defn*}

%% file: semiclassical.tex
\section{Semiclassical symmetry}
\label{sec:semiclassical}

\subsection{Semiclassical group actions}
\label{sec:semi}

Fix $d \in \ZZ^{\geq 0}$, and let $X$ be a $\pi$-finite space with cocycle $\tau \colon X
\to B^d \units{\bk}$.

\begin{defn}
A semiclassical action of a group $G$ on $\left(X , \tau\right)$ is 
a decorated correspondence of spaces
($1$-morphism in $\Fam_d\left(B^d \units{\bk}\right)$)
\begin{equation}
\begin{cd}
& \left(C_G , \tau_{G} \right) \ar{dr}\ar{dl}& \\
BG
&& * 
\end{cd}
\label{eqn:defn:CG}
\end{equation}
such that we have a pullback:
\begin{equation*}
\begin{cd}
&&\left(X , \tau\right) \ar{dl}\ar{dr}&&\\
& * \ar{dr}\ar{dl} && 
\left(C_G , \tau_{G} \right) \ar{dr}\ar{dl}& \\
* && 
BG
&& *
\end{cd}
\end{equation*}
\label{defn:semi}
\end{defn}

I.e. a semiclassical action of $G$ on $\left(X , \tau\right)$ is just an $X$-bundle over
$BG$, and an extension of $\tau$ to the total-space of this bundle.

\begin{exm}
Let $X = BL$ for a finite group $L$.
For trivial $\tau$, there is always a canonical semiclassical action of
$\Aut_\Grp\left(L\right)$ given by
\begin{equation*}
C_{\Aut} = B\left(L \rtimes \Aut_{\Grp}\left(L\right)\right)
\end{equation*}
since the following is a pullback diagram:
\begin{equation*}
\begin{cd}
&& BL\ar{dr}\ar{dl} &&  \\
& * \ar{dr}\ar{dl} && 
B\left(L\rtimes \Aut_\Grp\left(L\right)\right)
\ar{dr}\ar{dl}& \\
* && B\Aut_{\Grp}\left(L\right) && * 
\end{cd}
\end{equation*}
The same holds for any $\phi \colon G \to \Aut_\Grp\left(L\right)$.
Namely, 
\begin{equation*}
C_G = B\left(L\rtimes_\phi G\right) 
\end{equation*}
is the pullback
\begin{equation*}
\begin{cd}
&&
C_G
\ar{dr}\ar{dl}
&&
\\
&
BG 
\ar{dl}\ar{dr}{B\phi}
&&
C_{\Aut}
\ar{dl}\ar{dr}
&
\\
BG
&&
B\Aut_{\Grp}\left(L\right)
&& *
\end{cd}
\end{equation*}
and this defines a semiclassical action because we also have a pullback
\begin{equation*}
\begin{cd}
&& X \ar{dl}\ar{dr}&&\\
& * \ar{dr}\ar{dl} && 
C_G  \ar{dr}\ar{dl}& \\
* && 
BG
&& *
\end{cd}
\end{equation*}
\label{exm:semi}
\end{exm}

\begin{defn}
A \emph{projective} semiclassical action of a group $G$ on $\left(X , \tau\right)$ with
projectivity $c \colon BG \to B^{d+1} \units{\bk}$ is a correspondence of spaces
($1$-morphism in $\Fam_d\left(B^d \units{\bk}\right)$)
\begin{equation*}
\begin{cd}
& \left(C_G , \tau_{G} \right) \ar{dr}\ar{dl}& \\
\left(BG , c\right)
&& * 
\end{cd}
\end{equation*}
such that we have a pullback:
\begin{equation*}
\begin{cd}
&&\left(X , \tau\right) \ar{dl}\ar{dr}&&\\
& * \ar{dr}\ar{dl} && 
\left(C_G , \tau_{G} \right) \ar{dr}\ar{dl}& \\
* && 
\left(BG , c\right)
&& *
\end{cd}
\end{equation*}
\label{defn:proj_semi}
\end{defn}

\subsection{Spectral sequence obstruction}
\label{sec:ss}

Consider a semiclassical action $C_G$ of $G$ on $X$ as in \Cref{defn:semi} (i.e. the twist is
trivial).
Now, given some nontrivial $\tau$, we can pose the question of whether there is a semiclassical
action $\left(C_G , \tau_G\right)$ of $G$ on $\left(X , \tau\right)$. 
I.e. we are asking if 
the cocycle $\tau$ extends to a cocycle $\tau_G$ on $C_G$:
\begin{equation*}
\begin{tikzcd}
& B^d \units{\bk} & \\
X \ar{rr}\ar{d} \ar{ur}{\tau} &&
C_G \ar{d} \ar[dashed]{ul}{\tau_G} \\
* \ar{rr}&&
BG
\end{tikzcd}
\end{equation*}
The Serre spectral sequence provides a general method for answering this question.
The first quadrant Serre spectral sequence corresponding to the fibration
\begin{equation*}
\begin{cd}
X \ar{r}{\io}&
C_G \ar{d} \\
& BG 
\end{cd}
\end{equation*}
is 
\begin{equation*}
E^{p,q}_2 = H^p\left(BG , H^q\left(X , \units{\bk}\right)\right) \Rightarrow
H^{p+q}\left(C_G , \units{\bk}\right) \ .
\end{equation*}
Recall the differentials:
\begin{equation*}
d_r^{p,q} \colon E_r^{p,q} \to E_r^{p+r , q-r+1} \ .
\end{equation*}

The group $G$ acts on $H^d\left(X , \units{\bk}\right)$
(induced by conjugation in $\Om C_G$)
and every cocycle on $C_G$ is invariant under this action, so we insist that 
\begin{equation*}
\tau \in H^d\left(X , \units{\bk}\right)^G \ .
\end{equation*}
Then we have that 
\begin{equation*}
\tau \in \im\left(\io^*\right) = E^{0,d}_{d+2} \inj H^d\left(X , \units{\bk}\right)^G
\end{equation*}
if and only if
\begin{equation}
d_k^{0,d} \tau = d_k \tau = 0 \in E_k^{k , d+1-k}
\ .
\label{eqn:dk}
\end{equation}

\begin{theorem}
Assume \Cref{hyp:sigma}.
If the first $\left(d-1\right)$ obstructions \eqref{eqn:dk} vanish, i.e. $d_k \tau$ for
$2\leq k\leq d$, then there is a projective semiclassical action 
of $G$ on $\left(X , \tau\right)$, and this
quantizes to an anomalous action of $G$ on the theory $\sigma_{X,\tau}^d$ with anomaly
classified by $d_{d+1} \tau$.
\label{thm:ss}
\end{theorem}

\begin{proof}
Assuming the obstructions $d_k \tau$ vanish for $2\leq k \leq d$, 
there exists some class
\begin{equation*}
\tau_d \colon C_G \to B^{d} \units{\bk}
\end{equation*}
satisfying 
\begin{equation*}
d \tau_d = d_{d+1} \tau_d
\end{equation*}
so that we have a projective semiclassical action (\Cref{defn:proj_semi}):
\begin{equation*}
\begin{cd}
& \left(C_G , \tau_{G} \right) \ar{dr}\ar{dl}& \\
\left(BG , d_{d+1} \tau \right)
&& * 
\end{cd}
\end{equation*}
The quantization of this correspondence is a boundary theory:
\begin{equation*}
1 \to \sigma_{BG , d_{d+1} \tau}^{d+1}
\end{equation*}
such that pairing with the regular boundary condition gives a theory canonically
equivalent to $\sigma_{X,\tau}^d$.

By \Cref{prop:sigma_mod},
this is equivalent to an anomalous theory defined on $\Bord^{BG}$, with anomaly theory
classified by $d_{d+1} \tau \colon BG \to B^{d+1} \units{\bk}$.
\end{proof}

Recall our $d$-dimensional theories $\sigma_{X , \tau}^d$ are valued in $\Om \cT$ 
(where $\cT$ is fixed at the beginning of \Cref{sec:top_symm}).
Consider the $d$-type 
\begin{equation*}
B\Aut_\cT \left(\sigma_{X , \tau}^d \left(\pt\right)\right)
\end{equation*}
with top homotopy group $\pi_d$. 
Because we assumed $\Om^{d+1} \cT \simeq \units{\bk}$, we have a map:
\begin{equation}
\io \colon B^{d} \units{\bk} \to B^d \pi_d 
\ .
\label{eqn:io_scalars}
\end{equation}
Note that the target of $\io$ is the initial term of the Whitehead tower of
the $d$-type 
\begin{equation*}
B\Aut_\cT\left(\sigma^d_{X , \tau}\left(\pt\right)\right)
\ .
\end{equation*}
Write the $k$-invariant of 
$B\Aut_\cT\left(\sigma^d_{X , \tau}\left(\pt\right)\right)$
over the truncation $\pi_{\leq d-1}$ as
\begin{equation}
k \colon \pi_{\leq d-1} B\Aut_\cT\left(\sigma^d_{X , \tau}\left(\pt\right)\right)
\to B^{d+1} \pi_d \ .
\label{eqn:k_BAut}
\end{equation}
Unpacking \Cref{thm:ss} on the point gives the following. 

\begin{cor}
If the first $\left(d-1\right)$ obstructions \eqref{eqn:dk} vanish, then 
there is a map 
\begin{equation*}
\mathbf{f}\colon 
BG \to \pi_{\leq d-1} B\Aut_{\cT}\left(\sigma_{X , \tau} \left(\pt\right)\right)
\end{equation*}
and the projectivity of the projective $G$-action from \Cref{thm:ss} agrees with the
restriction of the $k$-invariant:
\begin{equation*}
\begin{tikzcd}
BG \ar{r}{\mathbf{f}}\ar{d}{d_{d+1}\tau} &
\pi_{\leq d-1} B\Aut_{\cT}\left(\sigma_{X , \tau} \left(\pt\right)\right)
\ar{d}{k} \\
B^{d+1} \units{\bk} \ar{r}{B\io}&
B \left(B\Om\right)^{d} B\Aut_{\cT} \left(\sigma_{X , \tau} \left(\pt\right)\right)
\end{tikzcd}
\end{equation*}
where $\io$ is from \eqref{eqn:io_scalars} and $k$ is from \eqref{eqn:k_BAut}.
\label{cor:ss}
\end{cor}

%% file: base.tex
\section{Symmetries of the base}
\label{sec:base}

Let $L$ be a finite group, and $\tau$ a group cocycle representing a class in
$H^3\left(L , \units{\bk}\right)$.
Let $\cC = \Vect\left[L\right]^\tau$ be the fusion category of vector spaces graded by
$L$, twisted by $\tau$.

In \Cref{sec:f}, we will construct an action by bimodules up to isomorphism:
\begin{equation*}
f \colon G \to \pi_0 \Aut_\Fus\left(\cC\right)
\end{equation*}
which is induced by an action of $G$ on the base:
\begin{equation*}
\phi \colon G \to \Aut\left(L , \tau\right)
\end{equation*}
where $\Aut\left(L , \tau\right)$ consists of group automorphisms of $L$ which preserve
the cohomology class of $\tau$.
This action can be twisted by a class $\gamma$ satisfying a certain equation. 

The map $f$ will lift to a map to the full $3$-group:
\begin{equation*}
\mathbf{f} \colon 
G\to \Aut_\Fus\left(\cC\right)
\end{equation*}
if and only if certain obstructions identified in \cite{ENO:homotopy} vanish, and are
trivialized. 
In \Cref{sec:m} and \Cref{sec:a}, we explicitly describe the data necessary to trivialize
these obstructions, which will be written as $\mu$ and $\alpha$ respectively. 
All together, given $L$ and $\tau$, the map $\mathbf{f}$ is determined by the map $\phi$
and the triple $\left(\gamma , \mu , \alpha\right)$.

Importantly, when the twist is trivial, the map $\phi$ canonically defines such a map
$\mathbf{f}$. This is the content of \Cref{thm:base}.

In \Cref{sec:semi_3d}, we describe a semiclassical action (\Cref{defn:semi}) associated to
the same data $\left(\phi , \gamma, \mu , \alpha\right)$. 
Furthermore, the quantization of this semiclassical action 
(i.e. image under the quantization functor \eqref{eqn:sum}) 
agrees with the map $\mathbf{f}$. 

\subsection{Action via bimodules up to isomorphism}
\label{sec:f}

Consider an arbitrary group homomorphism 
\begin{equation*}
\phi \colon G \to \Aut\left(L , \tau\right)
\end{equation*}
and a group $1$-cochain on $G$ valued in group $2$-cochains on $L$:
\begin{equation}
\gamma  \colon G \to C^2_\gp\left(L , \units{\bk} \right)
\ .
\label{eqn:gamma}
\end{equation}
If this satisfies the equation:
\begin{equation}
d_L\left(\gamma\left(g\right)\right) = \phi\left(g\right)\left(\tau\right)  \tau^{-1}
\ ,
\label{eqn:cond_gamma}
\end{equation}
then this defines a morphism in $\Fam_d$:
\begin{equation}
\begin{cd}
& \left(\Gamma\left(\phi\left(g\right)\right) , \gamma\left(g\right)\right)
\ar{dl}\ar{dr} & \\
\left(L , \tau\right) & & \left( L , \phi\left(g\right)\left(\tau\right)\right)
\end{cd}
\label{eqn:C_gamma}
\end{equation}
where $d_L$ is the group differential for $L$.
Note that such a $\gamma$ exists because the action $\phi$ 
preserves the cohomology class of $\tau$. 

Write $C_\gamma\left(g\right)$ for this $1$-morphism in $\Fam_d$.
Then we define:
\begin{equation}
f\left(g\right) \ceqq \Sum_d\left(C_\gamma\left(g\right)\right)
\label{eqn:f}
\end{equation}
where $\Sum_{d}$ is the quantization functor \eqref{eqn:sum}.

The $\cC$-bimodule $f\left(g\right)$ has the following description.
The underlying category is the same as $\cC$, and the left $\cC$-module structure is
simply the tensor product in $\cC$. 
The right action, say by some simple $\bk_\ell \in \cC$, is given by tensor product with 
$\bk_{\phi\left(\ell\right)}$.

The twist of this construction by the $2$-cocycle $\gamma\left(g\right)$ 
can be thought of as follows. 
Rather than defining $f\left(g\right)$ to be the trivial $\cC$-bimodule twisted by a group
automorphism of $L$, as above, we might consider twisting by a tensor autoequivalence
$\Phi$ of $\cC$. I.e. now an object $X\in \cC$ acts on the right by tensoring with
$\Phi\left(X\right)$.
The tensor autoequivalences form an extension:
\begin{equation*}
1 \to H^2\left(L , \units{\bk}\right) \to \pi_0 \Aut_\tp\left(\cC\right) \to \Aut\left(L ,
\tau\right) \to 1 \ ,
\end{equation*}
where $\Aut\left(L , \tau\right)$ denotes the group automorphisms of $L$ preserving the
cohomology class defined by the cocycle $\tau$.
See \cite[\S 11.2]{ENO:homotopy} for more.

\subsection{System of products}
\label{sec:m}

Now we would like to construct a system of products, i.e. equivalences:
\begin{equation*}
m\left(g,h\right) \colon f\left(h\right) \comp f\left(g\right) \lto{\sim} f\left(gh\right)
\end{equation*}
attached to all pairs $g,h\in G$.

We will obtain $m\left(g,h\right)$ by evaluating $\Sum_3$ on the 
correspondence between correspondences.
Note that for any $g,h\in G$ we have 
\begin{equation*}
\Gamma \left(\phi\left(g\right)\right) \comp \Gamma \left(\phi\left(h\right)\right)
= \Gamma \left(\phi\left(g\right) \comp \phi\left(h\right)\right)
= \Gamma\left(\phi\left(gh\right)\right)
\ .
\end{equation*}
so $f\left(h\right) \comp f\left(g\right)$ is the image of 
\begin{equation*}
\begin{tikzcd}
& \left(\Gamma\left(\phi\left(gh\right)\right) , \gamma\left(g\right) 
\gamma\left(h\right)\right) \ar{dr}\ar{dl}& \\
\left(L , \tau\right) && \left(L , \phi\left(gh\right) \left(\tau\right)\right)
\end{tikzcd}
\end{equation*}
under $\Sum_3$, and $f\left(gh\right)$ is the image of 
\begin{equation*}
\begin{tikzcd}
& \left(\Gamma\left(\phi\left(gh\right)\right) , \gamma\left(gh\right)\right) \ar{dl}\ar{dr}
& \\
\left(L , \tau\right) && \left(L , \phi\left(gh\right) \left(\tau\right)\right)
\end{tikzcd}
\end{equation*}
under $\Sum_3$.

Therefore we require a group 2-cochain on $G$ valued in $1$-cochains on $L$:
\begin{equation}
\mu \colon G^2  \to C^1_\gp\left(L , \units{\bk}\right)
\ .
\label{eqn:mu}
\end{equation}
If $\mu$ (and $\gamma$) satisfy:
\begin{equation*}
d_L \left(\mu\left(g , h\right)  \right) = 
\gamma\left(gh\right)
\gamma\left(g\right)^{-1} \gamma\left(h\right)^{-1} 
\end{equation*}
in $C^2\left(L , \units{\bk}\right)$
(where $d_G$ and $d_L$ are the differentials for the group cohomology of $G$ and $L$
respectively)
or more succinctly: 
\begin{equation}
d_L \mu = d_G \gamma^{-1} \ ,
\label{eqn:cond_mu}
\end{equation}
then this defines a $2$-morphism (correspondence of correspondences) in $\Fam_d$:
\begin{equation}
\begin{cd}
& \left(\Gamma\left(\phi\left(gh\right)\right) ,  \mu\left(g,h\right)\right)
\ar{dr}\ar{dl} & \\
\left(\Gamma\left(\phi\left(gh\right)\right) , \gamma\left(g\right)  \gamma\left(h\right)
\right)
&&
\left(\Gamma\left(\phi\left(gh\right)\right) , \gamma\left(gh\right) \right)
\end{cd}
\label{eqn:C_mu}
\end{equation}
Write $C_\mu\left(g,h\right)$ for this $2$-morphism in $\Fam_d$.
\begin{equation}
m\left(g,h\right) \ceqq \Sum_d\left(C_\mu\left(g,h\right)\right) \ .
\label{eqn:m}
\end{equation}

\subsubsection{The degree three obstruction}
\label{sec:base_O3}

Recall the obstruction $O_3\left(f\right) \in H^3_{\text{Grp}}\left(G , L\dsum
\pdual{L}\right)$ from \cite{ENO:homotopy}.
This is defined using $f$ and $m$, but only depends on $f$ \cite{ENO:homotopy}.
Recall $\gamma$ from \eqref{eqn:gamma} satisfying \eqref{eqn:cond_gamma}. 
Insisting that the obstruction $O_3\left(f\right)$ vanishes would put a further condition
on $\gamma$, however in the setting considered in this section, it will always vanish. 

Consider the cocycle:
\begin{equation*}
d_G \mu  \in Z^3_{\text{Grp}}\left(G , C^1\left(L\right)\right)\ .
\end{equation*}
In fact, this defines an element 
\begin{equation*}
d_G \mu \in 
Z^3_{\text{Grp}}\left(G , \pdual{L}\right)
\end{equation*}
where 
\begin{equation*}
\pdual{L} = \Hom\left(L , \units{\bk}\right) = Z^1\left(L , \units{\bk}\right)
\end{equation*}
because 
\begin{equation*}
d_L \comp d_G \left(\mu\right)= d_G \left( d_L \mu\right) = d_G \left(d_G \gamma
\right)^{-1} = 0
\end{equation*}
since $\gamma$ and $\mu$ satisfy \eqref{eqn:cond_mu}.

\begin{rmk}
This section is similar to the discussion of the same obstruction in 
\cite[\S 5.1]{CGPW}.
\end{rmk}

\begin{prop}
If $f$ is the group homomorphism defined by $\phi$ and $\gamma$ in \eqref{eqn:f}, then 
$O_3\left(f\right)$ vanishes.
\label{prop:base_O3}
\end{prop}

\begin{proof}
Recall that in this case the obstruction is in the following cohomology:
\begin{equation*}
O_3\left(f\right) \in H^3\left(G , Z\left(L\right)_{\text{Grp}} \dsum \pdual{L}\right)
\ .
\end{equation*}
As noted in \cite[\S 9.1]{ENO:homotopy}, this lies in the image of:
\begin{equation*}
H^3_{\text{Grp}}\left(G , Z\left(L\right)\right) \to H^3_{\text{Grp}}\left(G ,
Z\left(L\right)\dsum \pdual{L}\right)
\end{equation*}

The image of 
\begin{equation*}
\left[d_G \mu\right] \in H^3_{\text{Grp}}\left(G , \pdual{L} \right)
\end{equation*}
under the map 
\begin{equation*}
H^3_{\text{Grp}}\left(G , \pdual{L}\right) \to H^3_{\text{Grp}}\left(G ,  Z\left(L\right)
\dsum \pdual{L}\right)
\end{equation*}
induced by the inclusion $\pdual{L}\inj Z\left(L\right)\dsum \pdual{L} = \pi_2$ is the obstruction
$O_3\left(f\right)$ for $f$ defined in \eqref{eqn:f}.
I.e. $O_3\left(f\right)$ is trivializable.
\end{proof}

\subsection{Associators} 
\label{sec:a}

Now we would like to construct a system of associators, i.e. equivalences:
\begin{equation*}
a \left(g,h,k\right) \colon 
m\left(gh , k\right) \comp \left(m\left(g,h\right) \tp \id_{f\left(k\right)} \right) 
\lto{\sim} 
m\left(g , hk\right) \comp \left(\id_{f\left(g\right)} \tp m\left(h , k\right)\right)
\ .
\end{equation*}
The source of $a\left(g,h,k\right)$ will be the image of the following correspondence
(between correspondences between correspondences)
under $\Sum_3$:
\begin{equation*}
\begin{tikzcd}[row sep = small, column sep = 0]
&
\left(\Gamma\left(\phi\left(ghk\right)\right) , \mu\left(gh , k\right) \mu\left(gh\right)\right)
\ar{dr}\ar{dl}
&
\\
\left(\Gamma\left(\phi\left(ghk\right)\right) , \gamma\left(g\right) \gamma\left(h\right)
\gamma\left(k\right)\right) &&
\left(\Gamma\left(\phi\left(ghk\right)\right) , \gamma\left(ghk\right)  \right)
\end{tikzcd}
\end{equation*}
and the target will be the image of the following correspondence
(between correspondences between correspondences)
under $\Sum_3$:
\begin{equation*}
\begin{tikzcd}[row sep = small, column sep = 0]
&
\left(\Gamma\left(\phi\left(ghk\right)\right) , \mu\left(g,hk\right) \mu\left(hk\right)\right)
\ar{dr}\ar{dl}
&
\\
\left(\Gamma\left(\phi\left(ghk\right)\right) , \gamma\left(g\right) \gamma\left(h\right)
\gamma\left(k\right)\right) &&
\left(\Gamma\left(\phi\left(ghk\right)\right) , \gamma\left(ghk\right) \right)
\end{tikzcd}
\ .
\end{equation*}

Therefore we require a group $3$-cochain on $G$ valued in $C^0\left(L , \units{\bk}\right) =
\units{\bk}$:
\begin{equation}
\alpha \colon G^3 \to \units{\bk}
\ .
\label{eqn:alpha}
\end{equation}
The analogue of \eqref{eqn:cond_gamma} for $\gamma$ (and similarly of \eqref{eqn:cond_mu} for $\mu$) is
\begin{equation*}
d_L\left(\al\left(g,h,k\right)\right) 
= 
\mu\left(g , hk\right) \mu\left(h,k\right)
\mu\left(gh, k\right)^{-1} \mu\left(g,h\right)^{-1}
\end{equation*}
or more succinctly:
\begin{equation}
d_L \al = d_G \mu 
\label{eqn:cond_al}
\end{equation}
in $C^1\left(L , \units{\bk}\right)$. 
Since $\al\left(g,h,k\right)$ is a $0$-cocycle on $L$ (in cohomology for the trivial
action of $L$ on $\units{\bk}$), it is always closed. 
I.e. the condition reduces to 
\begin{equation}
d_G \mu = 0 \ .
\label{eqn:cond_al_0}
\end{equation}

The $3$-cochain $\alpha$ defines a correspondence (between correspondences of
correspondences):
\begin{equation*}
\begin{tikzcd}[row sep = small, column sep = -10pt]
&
\left(\Gamma\left(\phi\left(ghk\right)\right) , \al\left(g,h,k\right)\right)
\ar{dl}\ar{dr}
&
\\
\left(\Gamma\left(\phi\left(ghk\right)\right) , \mu\left(gh , k\right) \mu\left(gh\right)\right)
&&
\left(\Gamma\left(\phi\left(ghk\right)\right) , \mu\left(g,hk\right) \mu\left(hk\right)\right)
\end{tikzcd}
\label{eqn:C_alpha}
\end{equation*}
i.e. a $3$-morphism in $\Fam_d$, which we write as $C_\al\left(g,h,k\right)$. 
Finally, the associator is defined by:
\begin{equation}
a\left(g,h,k\right) \ceqq \Sum_d\left(C_\al \left(g,h,k\right)\right) \ .
\label{eqn:a}
\end{equation}

\subsubsection{Degree four obstruction}
\label{sec:base_O4}

Recall the obstruction $O_4\left(f , m\right) \in H^4_{\text{Grp}}\left(G , \units{\bk}\right)$. 
This is defined using $f$, $m$, and $a$, but only depends on $f$ and $m$
\cite{ENO:homotopy}.
When $f$ is the map defined in \eqref{eqn:f}, the system of products $m$ is the one
defined in \eqref{eqn:m}, and the associator $a$ is the one defined in \eqref{eqn:a}, we
have that the obstruction class is simply:
\begin{equation*}
O_4\left(f,m\right) = \left[d_G \alpha\right] \in H^4_{\text{Grp}}\left(G , \units{\bk}\right) \ ,
\end{equation*}
which always vanishes.

\subsubsection{Fully coherent action induced by an automorphism of the base}
\label{sec:coherent_base}

Let $L$ and $G$ be finite groups.
Let $\tau \in Z^3_{\text{Grp}}\left(L , \units{\bk}\right)$ be a $3$-cocycle in the group
cohomology of $L$ with coefficients in $\units{\bk}$, and consider an arbitrary group
homomorphism
\begin{equation*}
\phi \colon G\to \Aut\left(L , \tau\right) \ ,
\end{equation*}
where $\Aut\left(L , \tau\right)$ denotes the group of automorphisms preserving the
cohomology class of $\tau$.

\begin{theorem}
Assume \Cref{hyp:sigma}.
Let $L$ and $G$ be finite groups.
Let $\tau \in Z^3_{\text{Grp}}\left(L , \units{\bk}\right)$ 
be a $3$-cocycle in the group cohomology of $L$.
Then any group homomorphism 
\begin{equation*}
\phi \colon G\to \Aut\left(L , \tau\right)
\end{equation*}
canonically defines a
$\left(\sigma_{BG}^{4} , \rho\right)$-module structure on  $\sigma_{BL , \tau}^3$. 

This canonical module structure can be twisted by a triple
$\left(\gamma , \mu , \alpha\right)$ where 
\begin{align*}
\gamma \in C^1_{\text{Grp}} \left(G ,C^2_{\text{Grp}} \left(L, \units{\bk}\right) \right)
&&
\mu \in Z^2_{\text{Grp}}\left(G , C^1_{\text{Grp}} \left(L , \units{\bk} \right)\right)
&&
\alpha \in Z^3\left(G , \units{\bk}\right)
\ .
\end{align*}
satisfy:
\begin{align*}
d_L\left(\gamma\left(g\right)\right) = \phi\left(g\right)\left(\tau\right)  \tau^{-1}  &&
d_L \mu = d_G \gamma^{-1} 
\end{align*}
\label{thm:base}
\end{theorem}

\begin{proof}
The canonical structure will correspond to trivial $\gamma$, $\mu$, and $\alpha$. 
We will show that we obtain a module structure for arbitrary $\left(\gamma,  \mu ,
\alpha\right)$.

$\phi\left(g\right) \left(\tau\right) \tau^{-1}\in Z^2_{\text{Grp}} \left(L ,
\units{\bk}\right)$ defines the trivial cohomology class (since the action determined by
$\phi$ preserves the cohomology class of $\tau$). 
Therefore any $\gamma$ satisfies \eqref{eqn:cond_gamma}.
The fact that $O_3\left(f\right)$ for the group homomorphism $f$ defined by $\phi$ and
$\gamma$ follows from \Cref{prop:base_O3}.
As in \Cref{sec:base_O4}, the obstruction $O_4\left(f,m\right) = 0$ vanishes for all such
$\tau$ and $\phi$, and for all $\gamma$ and $\mu$.

Combining the group homomorphism $f$ from \eqref{eqn:f} 
(defined by $\phi$ and $\gamma$)
the system of products $m$ from \eqref{eqn:m} 
(defined by $\mu$)
and the associator $a$ from \eqref{eqn:a} 
(defined by $\al$), 
we obtain a map of $3$-groups:\footnote{Here $G$ denotes the discrete $2$-category with
objects $G$.}
\begin{equation}
\mathbf{f} = 
\left(f,m,a\right) \colon G \to \Aut_{\Fus}\left(\cC\right) \ .
\label{eqn:fma}
\end{equation}

By the cobordism hypothesis, this classifies a representation of 
$\Bord^{BG}_3$ in $\Fus$.
By \Cref{prop:sigma_mod}, this equivalently defines a $\left(\sigma_{BG}^{d+1} ,
\rho\right)$-module structure on $\sigma_{X , \tau}^d$.
\end{proof}

\begin{rmk}
We consider semiclassical actions 
$BG \from C_G  \to \pt$ of $G$ on $BL$ 
(\Cref{defn:semi}) of a very special form in this section:
\begin{equation*}
C_G = B\left(L\rtimes_\phi G\right) \ .
\end{equation*}
Arbitrary such semiclassical actions $C_G$ are equivalent to arbitrary group extension of
$G$ by $L$, which are classified by degree $2$ nonabelian group cohomology of $G$ with
coefficients in $L$. 

Even if $L$ is abelian, extensions given by semidirect products are still quite
restrictive:
For every $\phi$, there is an extension associated to every class
in $H^2_{\text{Grp} , \phi}\left(G , L\right)$, where this is group cohomology with 
the action $\phi$ of $G$ on $L$.

Such, more general, considerations are made in 
in \cite[\S 9.1,\S 11]{ENO:homotopy}. 
Namely, 
\begin{equation*}
\phi\left(gh\right)^{-1} \comp
\phi\left(g\right) \comp \phi\left(h\right)
\end{equation*}
might be conjugation by some element $n\left(g,h\right) \in L$.
This defines a class $\left[n\right] \in H^2\left(B G , L\right)$.

Once we have this twist $n$, the extension $L\rtimes_\phi G$
is replaced by the group extension determined by the pair 
$\left(\phi , n\right)$.
This introduces a possible nonzero obstruction class $O_4\left(f , m\right)$.
\label{rmk:ENOM}
\end{rmk}

\subsection{Semiclassical action}
\label{sec:semi_3d}

In this section we will construct a semiclassical action of $G$ on $\left(BL ,
\tau\right)$, from the same data as above: an action $\phi$ and a triple 
$\left( \gamma, \mu , \alpha\right)$.
The quantization of this semiclassical action 
(i.e. image under the quantization functor $\Sum_3$ in \eqref{eqn:sum}) 
will be equivalent to the map 
$\left(f,m,a\right)$ in \eqref{eqn:fma}. 

Note that the map $\phi$ canonically defines a semiclassical action of $G$ on $BL$ 
(\Cref{exm:semi}), and recall the anomaly/obstruction to the existence of a semiclassical
action of $G$ on $\left(BL , \tau\right)$ studied in \Cref{sec:ss}.
In this case, the space $X$ is $BL$, so the Serre spectral sequence considered in
\Cref{sec:ss} reduces to the Lyndon-Hochschild-Serre (LHS) spectral sequence associated to
the normal subgroup $L\inj L\rtimes_\phi G$:
\begin{equation*}
E^{p,q}_2 = H^p_{\text{Grp}}\left(G , H^q\left(L , \units{\bk}\right)\right) \Rightarrow
H^{p+q}_{\text{Grip}} \left(L\rtimes_\phi G , \units{\bk}\right)
\ .
\end{equation*}
See \cite{Mul:thesis,MS:anomalies,KT:discrete_anomalies} and \cite[\S 11.8]{ENO:homotopy}
for more on the LHS spectral sequence as a tool for encoding anomalies.

The following follows from \cite[\S 11.8]{ENO:homotopy}.

\begin{prop}
A map $\left(f,m,a\right)$ as in
\eqref{eqn:fma} equivalently defines an extension of the cocycle $\tau$ to 
a cocycle $\tau_G$ on $L\rtimes_\phi G$ which agrees with $\tau$ when restricted to $L$.
\label{prop:ENOM}
\end{prop}

Therefore, we have a $1$-morphism in the category $\Fam_4\left(B^4 \units{\bk}\right)$:
\begin{equation*}
\begin{cd}
& \left(B\left(L\rtimes_\phi G\right) , \tau_G \right)\ar{dr}\ar{dl} & \\
BG && \pt
\end{cd}
\end{equation*}
By \Cref{prop:corr_mor}, this induces a domain wall:
\begin{equation*}
F_G\colon
1 \to \sigma_{BG}^{d+1} \ ,
\end{equation*}
such that  when we pair with the regular boundary condition, we obtain
$\sigma_{BL}^{d+1}$. 

By \Cref{prop:sigma_mod}, $F_G$ is equivalent to a representation of $\Bord^{BG}_3$.
By the cobordism hypothesis, this is equivalent to a functor 
\begin{equation*}
F_G\colon BG \to B\Aut_\Fus\left(\cC\right) \ .
\end{equation*}
This functor $F_G$ provides another description of the functor $\left(f,m,a\right)$.

%% file: orthogonal.tex
\section{Orthogonal symmetries}
\label{sec:orthogonal}

\subsection{Canonical kernels}
\label{sec:BKS}

Let $\left(A , q\right)$ be a finite abelian metric group (i.e. a finite abelian group
equipped with nondegenerate quadratic form $q \colon A\to \units{\bk}$) with two
Lagrangian subgroup $L , K\subset A$. Then we have a correspondence of groups:
\begin{equation*}
\begin{cd}
& A / \left(L\cap K\right) 
\ar{dr}{ / \left(L / L\cap K\right)} 
\ar{dl}{ / \left(K / L\cap K\right)} &\\
A / K && A / L
\end{cd}
\end{equation*}
If $L\cap K = \left\{1\right\}$, i.e. they are transverse, then this takes the simpler
form:
\begin{equation}
\begin{cd}
& A \ar{dl}\ar{dr} & \\
L\simeq  A / K && A / L \simeq K
\end{cd}
\label{eqn:corr_trans}
\end{equation}

Decorating the correspondence \eqref{eqn:corr_trans} with any $2$-cocycle $\kappa$ on $A$
defines a $1$-morphism in $\Fam_3\left(B^3\units{\bk}\right)$ (recall \Cref{sec:sigma}).
We will primarily make use of the correspondence \eqref{eqn:corr_trans} when $K$ is the image
of $L$ under some $g\in \O\left(A , q\right)$. Continue to assume that $gL\cap L$ is
trivial, so that \eqref{eqn:corr_trans} decorated by $\kappa$ is the $1$-morphism 
\begin{equation*}
\begin{cd}
& \left(A , \kappa\right) \ar{dr}\ar{dl} & \\
L && gL
\end{cd}
\end{equation*}
in $\Fam_3\left(B^3\units{\bk}\right)$.

Assume there exists a polarization $A\simeq L\dsum \pdual{L}$, and that $q$ is given by
evaluation:
\begin{equation*}
\ev \colon A \simeq L\dsum \pdual{L} \to \units{\bk}
\ .
\end{equation*}
There is a canonical class in $H^2\left(L\dsum \pdual{L}\right)$ corresponding to the
Heisenberg extension. 
Let $\kappa_{\ev}$ be any $2$-cocycle representing the canonical class. 
For instance, one particular $2$-cocycle representing this class sends a pair of elements 
$\left(\ell , \phi\right)$ and $\left(k , \psi\right)$ of $L\dsum \pdual{L}$ to
$\psi\left(\ell\right)$.
This defines a $1$-morphism in $\Fam_3\left(B^3\units{\bk}\right)$:
\begin{equation}
\begin{cd}
& \left(A , \kappa_{\ev} \right)\ar{dl}{p_1}\ar{dr}{p_2} & \\
L && gL
\end{cd}
\label{eqn:corr_kappa}
\end{equation}
for any $g\in \O\left(L\dsum \pdual{L}\right)$ such that $gL\cap L$ is trivial.

\begin{rmk}
This cocycle $\kappa_{\ev}$ should be thought of as an analogue of the BKS
kernel in ordinary geometric quantization \cite{LV:Weil}.
\end{rmk}

\subsection{Higher intertwining operators}
\label{sec:intertwining}

The correspondences built in \Cref{sec:BKS} induce domain walls between the associated
TQFTs. Assume \Cref{hyp:sigma} so that there are well-defined TQFTs
\begin{equation*}
\sigma_{BL}^3 \colon \Bord_3^{\fra} \to \Fus 
\end{equation*}
for any finite abelian group $L$ defined using the functor $\Sum_3$. 

Let $g\in \O\left(L \dsum \pdual{L}\right)$ such that $gL \cap L$ is trivial. 
Applying \Cref{prop:corr_mor} to \eqref{eqn:corr_kappa}, we obtain a domain wall:
\begin{equation*}
\sigma_{BL}^3 \to \sigma_{B\left(gL\right)}^3  \ .
\end{equation*}
Fix some isomorphism
\begin{equation*}
\io \colon g L \simeq A / L = \pdual{L} \ .
\end{equation*}
Composing \eqref{eqn:corr_kappa} with $\io$ gives the following  correspondence:
\begin{equation}
\begin{cd}
& \left(A , \kappa_q \right) \ar{dr}{\io \comp p_2}\ar{dl}{p_1} & \\
L && \pdual{L}
\end{cd}
\label{eqn:corr_FT}
\end{equation}

\begin{defn}
The twice-categorified Fourier transform is the domain wall 
obtained by applying \Cref{prop:corr_mor} to the correspondence 
\eqref{eqn:corr_FT}:
\begin{equation*}
\FT \colon \sigma_{BL}^3 \to \sigma_{B\pdual{L}}^3 \ .
\end{equation*}
\label{defn:FT}
\end{defn}

\subsection{Evaluation on the point}
\label{sec:orthogonal_pt}

\subsubsection{As a
\texorpdfstring{$\Vect\left[L\right]$}{Vect[L]}-\texorpdfstring{$\Rep\left(L\right)$}{Rep(L)}-bimodule}

Recall from \Cref{exm:sigma_BG} (\Cref{fus}) the TQFT $\sigma_{BA}^3$ for any abelian
group sends the point to the fusion category $\Vect\left[A\right]$.
The natural transformation $\FT$, evaluated on the point, is therefore a
$\Vect\left[L\right]$-$\Vect\left[\pdual{L}\right]$-bimodule category. 
As it turns out (shown in \Cref{prop:FT_VectRep}) this is an
invertible bimodule category.

Recall that there is an equivalence of monoidal categories:
\begin{equation}
\Rep\left(L\right) \cong 
\Vect\left[\pdual{L}\right] 
\end{equation}
given by spectrally decomposing representations of $L$. 
Therefore $\FT\left(\pt\right)$ is equivalently an invertible 
$\Vect\left[L\right]$-$\Rep\left(L\right)$-bimodule category. 
Recall that here $\Rep\left(L\right)$ has the symmetric-monoidal tensor product, whereas
the monoidal structure on 
$\Vect\left[L\right]$ is given by convolution. In other words:

\begin{prop}
The twice-categorified Fourier transform (evaluated on a point)
is a Morita equivalence between
$\Vect\left[L\right]$ and $\Rep\left(L\right)$, 
i.e it exchanges convolution with the symmetric-monoidal tensor product. 
It also has a single simple object, i.e. is $\Vect$.
\label{prop:FT_VectRep}
\end{prop}

\begin{proof}
Recall the classification of $\Vect\left[L\right]$-$\Vect\left[\pdual{L}\right]$-bimodule
categories is equivalent to the classification of $\Vect\left[L\dsum
\pdual{L}\right]$-module categories. The latter are explicitly classified, e.g. in
\cite{ENO:homotopy} by subgroups of $L\dsum \pdual{L}$ equipped with alternating
bicharacters. 
In particular, the module category (and therefore bimodule category) corresponding to a
given subgroup $H$ has $\abs{L\dsum\pdual{L}} / \abs{H}$ simple objects. 
In this case the subgroup $H$ is the entire group, so there is a single simple object. 

The fact that this is an equivalence follows from the invertibility of the bimodule
category, which follows from \Cref{prop:FT_VectVect}.
\end{proof}

\begin{rmk}
The Morita equivalence in \Cref{prop:FT_VectRep} already appears in 
\cite[Example 7.12.19]{EGNO:TC} and \cite[Proposition 3.14]{FT:ising}.
\end{rmk}

\subsubsection{As a \texorpdfstring{$\Vect\left[L\right]$}{Vect[L]}-bimodule}

Choose some group isomorphism
\begin{equation}
\sigma \colon L\lto{\sim} \pdual{L}
\label{eqn:sigma}
\end{equation}
such that 
\begin{equation}
\sigma\left(\ell\right) \left(k\right) = \sigma\left(k\right) \left(\ell\right)
\label{eqn:symm}
\end{equation}
for all $k , \ell\in L$.
Define the corresponding orthogonal transform 
$T_\sigma \in \O\left(L\dsum \pdual{L}\right)$ by
\begin{equation*}
T_\sigma = 
\begin{pmatrix}
0 & \sigma^{-1} \\
\sigma & 0 
\end{pmatrix}
\end{equation*}
Note that this preserves the quadratic form because \eqref{eqn:symm} is satisfied:
\begin{equation*}
\ev\left(
T_\sigma\left(\ell , \phi\right)
\right)
= \sigma\left(\ell\right) \left[\sigma^{-1}\left(\phi\right)\right]
= \sigma\left(\sigma^{-1} \phi\right) \left(\ell\right) 
= \phi\left(\ell\right) = \ev\left(\ell , \phi\right)
\end{equation*}

\begin{rmk}
This is a noncanonical choice!
For any abelian group $L$, such 
a symmetric equivalence $\sigma$ always exists, which can be seen 
e.g. componentwise in the primary decomposition of $L$.
Furthermore, the orthogonal operators $T_\sigma$ and $T_{\sigma'}$
defined by any two such choices are related by conjugation.\footnote{Explicitly, the
automorphism of $L$ given by $\sigma^{-1}\comp \sigma'$ defines an orthogonal
transformation of $L\dsum \pdual{L}$. Conjugation by this element relates $T_\sigma$ to
$T_\sigma'$.}
\end{rmk}

Composing the correspondence \eqref{eqn:corr_FT}
with the equivalence \eqref{eqn:sigma}, we obtain a correspondence:
\begin{equation}
\begin{cd}
& \left(L\dsum \pdual{L} , \kappa_{\ev} \right) \ar{dr}{\sigma^{-1}\comp \iota \comp p_2} 
\ar{dl}{p_1} & \\
L && L
\end{cd}
\label{eqn:corr_LL}
\end{equation}
Equivalently, we can consider the correspondence:
\begin{equation}
\begin{cd}
& \left(L\dsum L , \kappa_{L} \right) \ar{dr}{p_2} 
\ar{dl}{p_1} & \\
L && L
\end{cd}
\label{eqn:corr_LL2}
\end{equation}
where $\kappa_{L}$ is a $2$-cocycle on $L^2$ representing the class
in $H^2_{\text{Grp}}\left(L^2 , \units{\bk}\right)$ corresponding to the 
skew-symmetric bicharacter $\alt \kappa_L$ (see \Cref{sec:schur}) sending 
\begin{equation}
\left(\ell_1 , k_1\right)  , \left(\ell_2 , k_2\right)
\mapsto 
\sigma\left(k_2\right)\left(\ell_1\right)^{-1}
\sigma\left(k_1\right)\left(\ell_2\right)
\ .
\label{eqn:kappa_L}
\end{equation}
We obtain this formula by post-composing the 
$2$-cocycle $\kappa_{\ev}$ with $\sigma$ and
the inversion map, following by the alternating map from 
\eqref{eqn:alt_equiv} to obtain the alternating bicharacter. 
Write $C_{L,L}$ for the correspondence in \eqref{eqn:corr_LL}
(or equivalently \eqref{eqn:corr_LL2}).

Recall the quantization map $\Sum_3$ from \Cref{hyp:sigma},
in particular the explicit description of $\Sum_3$ in 
\Cref{exm:sigma_BG} (\Cref{fus,bimodule_cats}).

\begin{prop}
The bimodule category $\Sum_3\left(C_{L,L}\right)$ has a single simple object, 
and is of order $2$. 
\label{prop:FT_VectVect}.
\end{prop}

\begin{proof}
The bimodule category $\Sum_3\left(C_{L,L}\right)$ 
corresponds to a $\Vect\left[L^2\right]$-module category under the usual correspondence
described in \Cref{sec:schur}.

Recall that module categories over $\Vect\left[L^2\right]$ are 
are classified by subgroups of $L^2$ equipped with alternating bicharacters. 
The subgroup corresponding to $\Sum_3\left(C_{L,L}\right)$
is the whole group $L^2$, which is enough to see that the (bi)module category
has a single simple object, since simple objects are in
correspondence with cosets of the subgroup in $L^2$ \cite[\S 2.7]{ENO:homotopy}.

The alternating bicharacter is $\alt\kappa_L$ defined in \eqref{eqn:kappa_L}.
Alternating bicharacters on an abelian group are equivalent to degree $2$ cohomology
classes on the group (see \Cref{sec:schur}).
Recall the Schur isomorphism (\eqref{eqn:schur}) in this case:
\begin{equation}
H^2\left(L^2\right) \simeq H^2\left(L\right) \dsum H^2\left(L\right) \dsum 
\Hom\left(L \tp_\ZZ L, \units{\bk}\right) 
\end{equation}
where the third factor consists of bicharacters on $L$. 
The cohomology class of $\kappa_L$ is trivial in the first two factors, and $\sigma$ in
the third factor.

By \Cref{lem:mod_bimod}, the bimodule therefore has the following explicit description. 
The underlying category is $\Vect$ and the action of $\Vect\left[L\right]$ on either side
is trivial. The only nontrivial data is the middle associator,
a bicharacter on $L$ sending:
\begin{equation*}
\left(\ell , k\right) \mapsto \sigma\left(k\right) \left(\ell\right) \ .
\end{equation*}
As is explained in \cite[Prop. 9.3]{ENO:homotopy}, this bimodule category is invertible on
account of the fact that $\sigma$ is an isomorphism, and it is order $2$ on account of the
fact that $\sigma$ satisfies \eqref{eqn:symm}.
\end{proof}

\subsection{Full orthogonal group}
\label{sec:orthogonal_full}

It is shown in \cite{ENO:homotopy} that
the group of invertible $\Vect\left[L\right]$-bimodules 
is equivalent to the group $\O\left(L\dsum \pdual{L}\right)$.
This equivalence is given a description in \cite[\S 10]{ENO:homotopy} via
the hyperbolic part of the category of Lagrangian correspondences.

Namely, they assign the following $\Vect\left[L^2\right]$-module category to $g\in
\O\left(L\dsum \pdual{L} , \ev\right)$. 
Recall $\Vect\left[L^2\right]$-module categories are classified by a subgroup $H\subset
L^2$ and a class in $H^2\left(L^2 , \units{\bk}\right)$, or equivalently 
an alternating bicharacter on $L^2$ (see \Cref{sec:schur}).
The subgroup $H$ is the projection of the graph of $g$ to $L^2 \subset \left(L\dsum
\pdual{L}\right)^2$.
The alternating bicharacter $\chi_g$ is defined by:
\begin{equation*}
\chi_g\left( \left(\ell_1 , k_1\right) , \left(\ell_2 , k_2\right)\right)
\ceqq 
\lr{ \widetilde{ \left(\ell_1 , k_1\right)} , \left(\ell_2 , 1 , k_2 ,
1\right)}_{q} 
\end{equation*}
where $\widetilde{\left(\ell_1 , k_2\right)}$ denotes an arbitrary lift of $\left(\ell_1 ,
k_2\right)$ to the graph of $g$ in $\left(L\dsum \pdual{L}\right)^2$
and $q = \ev_A \dsum \ev_A^{-1}$. 

In particular, this bicharacter works out to be:
\begin{align*}
\chi_g\left( \left(\ell_1 , k_1\right) , \left(\ell_2 , k_2\right)\right)
&= \lr{ \widetilde{\left(\ell_1 , k_1\right)} , \left(\ell_2 , 1 , k_2 , 1\right) }_{q} \\
&= q \left(\ell_1 + \ell_2 , \phi , k_1 + k_2 , \psi\right) \\
&= \phi\left(\ell_2\right) \psi\left(k_2\right)^{-1}
\ ,
\end{align*}
where $\phi \in \pdual{L}$ is such that 
\begin{equation*}
k_1 = p_3  g\left(\ell_1 , \phi\right) \ ,
\end{equation*}
where $p_3$ denotes the projection from the graph of $g$ to the third factor $L$, and 
\begin{equation*}
\psi \ceqq p_4 g\left(\ell_1 , \phi\right) \ .
\end{equation*}

Recasting this in $\Fam_3\left(B^3 \units{\bk}\right)$, 
there is a correspondence $C_g$ defined for every element of 
$g\in \O\left(L\dsum \pdual{L}\right)$:
\begin{equation}
C_g = 
\begin{tikzcd}
& \left(p_{1,3} \Gamma\left(g\right) , \kappa_g \right) \ar{dl}\ar{dr} & \\
L && L
\end{tikzcd}
\label{eqn:C_g}
\end{equation}
where $\kappa_g$ is a $2$-cocycle representing a class in 
$H^2\left(p_{1,3} \Gamma\left(g\right)\right)$ which corresponds to the alternating
bicharacter $\chi_g$ on $p_{1,3} \Gamma\left(g\right)$ defined above.

Define a functor 
\begin{equation}
f \colon \O\left(L \dsum \pdual{L}\right) \to \pi_0
\Aut_{\Fus}\left(\Vect\left[L\right]\right) 
\label{eqn:f_ENO}
\end{equation}
by sending $g$ to $\Sum_3\left(C_g\right)$. 
Recall (from \Cref{exm:sigma_BG} \Cref{fus,bimodule_cats})
that the value of $\Sum_3$ is given precisely by the 
equivalence given in \cite{ENO:homotopy} between  $\O\left(L\dsum \pdual{L}\right)$ and
the group of isomorphism classes of invertible $\Vect\left[L\right]$-bimodule categories.

\begin{theorem}
The map $f$ \eqref{eqn:f_ENO} sends the order $2$ element $T_\sigma$ to $\FT$
(\Cref{defn:FT}).
\label{thm:ENO}
\end{theorem}

\begin{proof}
It suffices to show that the correspondence $C_g$ in \eqref{eqn:C_g} and $C_{L,L}$ in 
\eqref{eqn:corr_LL2} are equivalent $1$-morphisms in $\Fam_3\left(B^3\units{\bk}\right)$,
i.e. we just have to show they are decorated by the same $2$-cocycle, or equivalently the
same bicharacter. 

As above, the bicharacter $\chi_g$ corresponding to $C_g$ (the image under 
\eqref{eqn:alt_equiv} of $\kappa_g$) is:
\begin{equation*}
\chi_g\left( \left(\ell_1 , k_1\right) , \left(\ell_2 , k_2\right)
\right)= \phi\left(\ell_2\right) \psi\left(k_2\right)  \ .
\end{equation*}
When $g = T_\sigma$, a particular choice of lift is given by
\begin{equation*}
\phi = 
\sigma\left(k_1\right)
\qquad 
\psi = \sigma\left(\ell_1\right)
\end{equation*}
making the formula:
\begin{equation*}
\chi_{T_{\sigma}} \left( \left(\ell_1 , k_1\right) , \left(\ell_2 , k_2\right)
\right)= \sigma\left(k_1\right)\left(\ell_2\right)
\sigma\left(\ell_1\right)\left(k_2\right)^{-1}  \ .
\end{equation*}

On the other hand, recall that the bicharacter defined in 
\eqref{eqn:kappa_L}, corresponding to the cocycle decorating $C_{L,L}$
sends:
\begin{equation*}
\left(\ell_1 , k_1\right) , \left(\ell_2 , k_2\right)
\quad \mapsto \quad 
\sigma\left(k_2\right)\left(\ell_1\right)^{-1}
\sigma\left(k_1\right)\left(\ell_2\right)
 \ .
\end{equation*}
On account of \eqref{eqn:symm} we have 
\begin{equation*}
\sigma\left(k_2\right)\left(\ell_1\right) = \sigma\left(\ell_1\right)\left(k_2\right) \ ,
\end{equation*}
and therefore the bicharacters agree.
\end{proof}

\subsection{Field-theoretic interpretation}
\label{sec:orthogonal_TFT}

As is explained in \cite{VD:3d}, the results of \cite{ENO:homotopy} and
\Cref{prop:sigma_mod} together imply that the group of orthogonal automorphisms of
$A\simeq L\dsum \pdual{L}$ canonically induce projective automorphisms of the theory
$\sigma_{BL}^3$. In other words: 
Twice-categorified geometric quantization of odd-order metric groups is projectively functorial. 

\begin{theorem}
Assume \Cref{hyp:sigma} and that $L$ has odd order.
The canonical projective action of $T_\sigma\in \O\left(L\dsum \pdual{L}\right)$ 
on $\sigma_{BL}^3$ is via the domain wall $\FT$ (\Cref{defn:FT}).
\label{thm:FT}
\end{theorem}

\begin{proof}
Recall that the canonical projective action 
of $\O\left(L\dsum \pdual{L}\right)$ on $\sigma_{BL}^3$ is defined as follows.
In \cite[\S 5.2]{ENO:homotopy} an equivalence 
\begin{equation*}
\Psi \colon 
\Aut_\EqBr\left(\cZ\left(\cC\right)\right)
\lto{\sim} \pi_{\leq 1} \Aut_{\Fus}\left(\cC\right)
\end{equation*}
is provided for any fusion category $\cC$.
When $\cC = \Vect\left[L\right]$ and $\abs{L}$ is odd
there is a canonical splitting
(\cite[\S 6]{EG:reflection} and \cite[\S 5]{CGPW}):
\begin{equation*}
s \colon \O\left(L\dsum \pdual{L} , \ev\right) =
\pi_{0} \Aut_\EqBr\left(\cZ\left(\cC\right)\right)
\to
\pi_{\leq 1} \Aut_\EqBr\left(\cZ\left(\cC\right)\right)
\ .
\end{equation*}
Composing these maps yields a canonical action of $\O\left(L\dsum \pdual{L}\right)$ on 
$\Vect\left[L\right]$ regarded as an object of $\PP\Fus$. 
By the cobordism hypothesis this classifies a projective action of $\O\left(L\dsum
\pdual{L}\right)$ on $\sigma_{BL}^3$. 
Therefore the theorem follows from the cobordism hypothesis and 
\Cref{prop:FT_VectVect}.
\end{proof}

%% file: schur.tex
\section{Bicharacters and the Schur multiplier}
\label{sec:schur}

Let $\ext^2\left(-\right)$ denote the group of alternating bicharacters, and let
$Z^2\left(- , \units{\bk}\right)$ denote the group of degree two group cocycles.
Let $A$ be an abelian group. 
There is a map
\begin{equation*}
\alt \colon Z^2\left(A , \units{\bk}\right) \to \ext^2\left(A\right)
\end{equation*}
given by sending a cocycle $\al \in Z^2\left(A\right)$ to the alternating bicharacter
sending:
\begin{equation*}
\left( a_1 , a_2\right) \mapsto \al\left(a_1 , a_2\right) \al\left(a_2 , a_1\right)^{-1} 
\ .
\end{equation*}
It is a well-known fact that this induces an equivalence:
\begin{equation}
\alt \colon H^2\left(A , \units{\bk} \right) \lto{\sim} \ext^2 A \ .
\label{eqn:alt_equiv}
\end{equation}
See e.g. \cite[2.4]{NN:Lag}.

If $A = A_1 \dsum A_2$ is an abelian group, recall the \emph{Schur isomorphism}:
\begin{equation}
H^2\left(A\right) \simeq H^2\left(A_1\right) \dsum H^2\left(A_2\right) \dsum \Hom\left(A_1
\tp_\ZZ A_2, \units{\bk}\right) \ .
\label{eqn:schur}
\end{equation}
See e.g. \cite[Theorem 2.2.10]{Karp:Schur}.

There is an equivalence of $\Vect\left[A_1\dsum A_2\right]$-modules and 
$\Vect\left[A_1\right]$-$\Vect\left[A_2\right]$-bimodules
sending a module to the same underlying category, with bimodule structure:
\begin{equation*}
\bk_\ell \tp \left(-\right) \tp \bk_k = \bk_{\left(\ell , -k\right)}\tp \left(-\right) 
\end{equation*}
where $\ell\in A_1$ and $k\in A_2$.

\begin{lem}
Under this equivalence, $\Vect$ with associator 
$\left[\al\right] \in H^2\left(A_1 \dsum A_2\right)$ gets sent to $\Vect$ with middle associator
\begin{equation*}
\left(\ell , k\right) \mapsto 
\left(\alt \al\right)
\left( \left(\ell , 0\right) , \left(0 , k\right)\right)
\ .
\end{equation*}
\label{lem:mod_bimod}
\end{lem}

\begin{proof}
Let $\al \in Z^2\left(A_1 \dsum A_2\right)$ be the associator for a $\Vect\left[A_1 \dsum
A_2\right]$-module category with underlying category $\Vect$.
The middle associator for the associated
$\Vect\left[A_1\right]$-$\Vect\left[A_2\right]$-bimodule category is an equivalence, for
all $\ell \in A_1$ and $k\in A_2$ between
\begin{equation*}
\left(\bk_{\ell} \tp -\right) \tp \bk_{k} = 
\bk_{ \left(0 , -k\right) } \tp 
\left( \bk_{ \left(\ell , 0\right)} \tp - \right)
\end{equation*}
and 
\begin{equation*}
\bk_{\ell} \tp 
\left(- \tp \bk_{k} \right)
=
\bk_{ \left(\ell , 0\right) } \tp 
\left( \bk_{ \left(0 , -k\right)} \tp - \right)
\ .
\end{equation*}

Because $A$ is abelian, this is provided by the following composition:
\begin{equation*}
\begin{tikzcd}
\bk_{ \left(0 , -k\right) } \tp 
\left( \bk_{ \left(\ell , 0\right)} \tp - \right)
\ar{d}{\al\left( \left(0 , -k\right) , \left(\ell , 0\right)\right)^{-1}}
\\
\left( \bk_{ \left(0 , -k\right) } \tp 
\bk_{ \left(\ell , 0\right)} \right) \tp \left(-\right)
\ar{d}\\
\left( \bk_{ \left(\ell , 0\right) } \tp 
\bk_{ \left(- , -k\right)} \right) \tp \left(-\right)
\ar{d}{\al\left( \left(\ell , 0\right) , \left(0 , -k\right)\right)}
\\
\bk_{ \left(\ell , 0\right) } \tp 
\left( \bk_{ \left(0 , -k\right)} \tp - \right)
\end{tikzcd}
\end{equation*}
\end{proof}